\documentclass[11pt,a4paper,reqno]{amsart}

\usepackage[english]{babel}
\usepackage[latin1]{inputenc}
\usepackage[hidelinks]{hyperref}
\usepackage{math rsfs,
xcolor,
dsfont
}
\usepackage{amsmath}
\usepackage[english]{babel}
\usepackage{latexsym}
\usepackage{amssymb}
\usepackage{amscd}
\usepackage{amsgen,amstext,amsbsy,amsopn,amsfonts}
\usepackage{math rsfs}
\usepackage{bm,bbm}
\usepackage{amsthm,epsfig,graphicx,graphics}
\usepackage{xspace}
\usepackage{amsxtra}
\usepackage{xcolor}
\usepackage{enumerate}
\usepackage{tikz}

\usepackage{geometry}

\makeatletter
\renewcommand*{\@textcolor}[3]{%
  \protect\leavevmode
  \begingroup
    \color#1{#2}#3%
  \endgroup
}
\makeatother

\numberwithin{equation}{section}

\newtheorem{teo}{Theorem}[section]
\newtheorem{lem}{Lemma}[section]
\newtheorem{pro}{Proposition}[section]
\newtheorem{cor}{Corollary}[section]
\newtheorem{defi}{Definition}[section]
\theoremstyle{remark}
\newtheorem{rem}{Remark}[section]
\newcommand{\bdm}{\begin{displaymath}}
\newcommand{\edm}{\end{displaymath}}
\newcommand{\bay}{\begin{array}{c}}
\newcommand{\eay}{\end{array}}
\newcommand{\ben}{\begin{enumerate}}
\newcommand{\een}{\end{enumerate}}
\newcommand{\beq}{\begin{equation}}
\newcommand{\eeq}{\end{equation}}
\newcommand{\beqn}{\begin{eqnarray}}
\newcommand{\eeqn}{\end{eqnarray}}
\newcommand{\bml}[1]{\begin{multline} #1 \end{multline}}
\newcommand{\bmln}[1]{\begin{multline*} #1 \end{multline*}}

\renewcommand{\leq}{\leqslant}
\renewcommand{\geq}{\geqslant}

\newcommand{\tx}{\textstyle}
\newcommand{\lf}{\left}
\newcommand{\ri}{\right}
\newcommand{\R}{\mathbb{R}}
\newcommand{\C}{\mathbb{C}}
\newcommand{\M}{\mathcal{M}}

\newcommand{\DE}{\mathscr{D}}
\newcommand{\DES}{\DE_{\mathcal{S}}}
\newcommand{\nabk}{\nabla_{\kv}}

\newcommand{\Q}{\mathcal{Q}}

\newcommand{\I}{\mathcal{I}}
\newcommand{\ham}{\mathcal{H}_\alpha}
\newcommand{\SC}{\mathcal{S}(\R^2)}
\newcommand{\ep}{\varepsilon}
\newcommand{\psitr}{\widehat{\psi}}
\newcommand{\psitrs}{\widehat{\psi}_t^*}
\newcommand{\phitr}{\widehat{\phi}}
\newcommand{\ftr}{\widehat{f}}

\newcommand{\PI}{\mathrm{Im}}
\newcommand{\PR}{\mathrm{Re}}

\newcommand{\one}{\mathds{1}}
\newcommand{\palla}{{k \leq R}}
\newcommand{\bpalla}{{k = R}}
\newcommand{\h}{h_T}
\newcommand{\T}{\widetilde{T}}
\newcommand{\NN}{\mathcal{N}}
\newcommand{\B}{\mathcal{B}}

\newcommand{\zev}{\mathbf{0}}
\newcommand{\xv}{\mathbf{x}}
\newcommand{\kv}{\mathbf{k}}
\newcommand{\yv}{\mathbf{y}}
\newcommand{\f}{\frac}
\newcommand{\tf}{\tfrac}

\newcommand{\dtau}{\,\mathrm{d}\tau}
\newcommand{\ds}{\,\mathrm{d}s}
\newcommand{\dt}{\,\mathrm{d}t}
\newcommand{\dH}{\,\mathrm{d}\Sigma}
\newcommand{\dkv}{\,\mathrm{d}\kv}
\newcommand{\dxv}{\,\mathrm{d}\xv}
\newcommand{\dyv}{\,\mathrm{d}\yv}

\newcommand{\deta}{\,\mathrm{d}\eta}
\newcommand{\diff}{\mathrm{d}}

\begin{document}
 
 \title[Blow-up for 2D NLS with concentrated nonlinearity]{Blow-up for the pointwise NLS in dimension two: \\ absence of critical power}

\author[R. Adami]{Riccardo Adami}
\address{Politecnico di Torino, Dipartimento di Scienze Matematiche ``G.L. Lagrange'', Corso Duca degli Abruzzi, 24, 10129, Torino, Italy.}
\email{riccardo.adami@polito.it}
\author[R. Carlone]{Raffaele Carlone}
\address{Universit\`{a} degli Studi di Napoli ``Federico II'', Dipartimento di Matematica e Applicazioni ``R. Caccioppoli'', MSA, via Cinthia, I-80126, Napoli, Italy.}
\email{raffaele.carlone@unina.it}
\author[M. Correggi]{Michele Correggi}
\address{``Sapienza'' Universit\`{a} di Roma, Dipartimento di Matematica, P.le Aldo Moro, 5, 00185, Roma, Italy.}
\urladdr{http://www1.mat.uniroma1.it/people/correggi/}
\email{michele.correggi@gmail.com}
\author[L. Tentarelli]{Lorenzo Tentarelli}
\address{``Sapienza'' Universit\`{a} di Roma, Dipartimento di Matematica, P.le Aldo Moro, 5, 00185, Roma, Italy.}
\email{tentarelli@mat.uniroma1.it}

\date{\today}

\begin{abstract} 
 {We consider the Schr\"odinger equation in dimension two with a fixed, pointwise, focusing nonlinearity and show the occurrence of a blow-up phenomenon with two peculiar features: first, the energy threshold under which all solutions blow up is strictly negative and coincides with the infimum of the energy of the standing waves. Second, there is no critical power nonlinearity, i.e. for every power there exist blow-up solutions. This last property is uncommon among the conservative Schr\"odinger equations with local nonlinearity.}
\end{abstract}

\maketitle


\section{Introduction}

The introduction of concentrated nonlinearities for the Schr\"odinger equation dates back to the nineties of the last century  \cite{bonilla,malomed,nier}. It was motivated by the need for modeling the effect of a nonlinear centre on a quantum particle, under the assumption that the size of the centre is smaller than the wave-length of the particle. In turn, the nonlinear centre was understood as an effective description, through a suitable scaling limit, of
a cluster of a large number of particles confined in a small region of space \cite{jona}.

The issues of well-posedness and globality of solutions were investigated in \cite{AT} for the problem in one dimension and in \cite{ADFT1,ADFT2} for the one in three dimensions; while the derivation from the standard NonLinear Schr\"odinger Equation (NLSE), i.e.,  the Schr\"odinger equation with a nonlinear term of  form $f (|\psi|) \psi$ with $f$ real-valued, is due to \cite{CFNT1,CFNT}.

It turned out that the NLSE with pointwise nonlinearity shares some specific features with the standard NLSE: in particular, conservation of mass and energy holds and the globality of all solutions in the energy space is guaranteed, provided that the  nonlinearity is defocusing. A
blow-up phenomenon emerges in the case of focusing nonlinearity.

\noindent More precisely,  blow-up solutions can occur only if the growth rate at infinity of the nonlinear term is not slower than a specific power law, that defines the {\em critical power} of the problem. 

In the present paper, we show that the parallelism with the standard NLSE breaks for NLSE with pointwise nonlinearity in dimension two.

Preliminarily, let us recall that the exotic properties of the NLSE in two dimensions with pointwise nonlinearity had already emerged in the issue of the rigorous set-up of the problem \cite{CCT}.

\noindent 
More strikingly, in the present work we show that the blow-up phenomenon does not mimic its analogue for the standard NLSE under several aspects.  The most remarkable is the {\em absence of a critical power}: namely, for every superlinear power growth of the nonlinear focusing term, some solutions blow up in finite time.

\noindent
To our knowledge, this is the only known model of NLSE with a local and conservative
nonlinearity that exhibits such  behaviour, already observed, on the other hand,
for some non-conservative Schr\"odinger equations \cite{ikeda1,ikeda2,ikeda3}. Furthermore, like in the standard case, every initial datum with sufficiently low energy blows up and the energy threshold coincides with the infimum of the energy of the stationary solutions; but, contrarily to the standard case, such a threshold turns out to be strictly negative and finite for every nonlinearity power. 

In the present paper we prove all these facts by using the classical virial method due to Glassey \cite{Glassey}.  We show that, if a solution lies below a given energy threshold, then its moment of inertia is strictly concave and this prevents the solution from existing for an arbitrarily large time.

\noindent Of course, the energy threshold gives a sufficient condition only and does not provide information either on the shape of blow-up solutions or on the blow-up time rate. We do not see any reason for these to be similar to those discovered for the standard NLS \cite{perelman,merle1,merle2,merle3}, and we plan to further investigate this problem in a future work. 

We recall that a thorough analysis of the blow-up  for a one-dimensional NLSE with concentrated nonlinearity has been carried out in \cite{holmer1,holmer2}, while the interplay between standard nonlinearity and linear delta potential has been studied in \cite{banica} in the scattering context.

\noindent Finally, we mention that recently the issue of the pointwise nonlinearities has been also discussed in the context of  quantum beating  \cite{carlone}, in that of the fractional Schr\"odinger equation \cite{fract} and in that of the Dirac equation \cite{CCNP}. On the other hand, some recent works manage more general singular problems, such as \cite{G-CV, OP, R}.


\subsection{Organization of the paper}

The paper is organized as follows:

\begin{itemize}
 \item[(i)] in Section \ref{sec-set} we introduce the model of the point interactions in dimension two, fix the functional setting of the problem (i.e., \eqref{eq:NLSEweak}) and present the main results of the paper:
 \begin{itemize}
  \item[-] Theorem \ref{teo:sufficient}, which provides a sufficient condition for the existence of blow-up solutions,
  \item[-] Theorem \ref{teo-stand}, which classifies all the standing waves of the problem;
 \end{itemize}
 \item[(ii)] in Section \ref{sec:standing} we show the proof of Theorem \ref{teo-stand};
 \item[(iii)] in Section \ref{sec:prel} we introduce the moment of inertia $ \M(t)$ and present a heuristic justification of the formulae for its first (Proposition \ref{pro:first_der}) and second (Proposition \ref{pro:second_der}) derivatives leading to the proof of Theorem \ref{teo:sufficient};
 \item[(iv)] in Section \ref{app} we exhibit a rigorous proof of the formulae for the first and second derivatives of the moment of inertia.
\end{itemize}

\bigskip
\bigskip
\noindent
{\bf Acknowledgements.} R.C., M.C. and L.T. acknowledge the support of MIUR through the FIR grant 2013 ``Condensed Matter in Mathematical Physics (Cond-Math)'' (code RBFR13WAET).
R.A. acknowledges the project PRIN 2015 ``Variational methods with applications to problems in mathematical physics and geometry'', funded by MIUR.
\bigskip


\section{Setting and main results}
\label{sec-set}

The evolution problem we aim at studying can be formally introduced as:
\begin{equation}
 \label{eq:NLSE}
 \left\{
 \begin{array}{l}
  \displaystyle i\partial_t \psi_t= - \Delta \psi_t - \beta | \psi_t |^{2 \sigma} \delta_0 \psi_t\\[.2cm]
  \displaystyle \psi_{t=0}=\psi_0,
 \end{array}
 \right.
\end{equation}
where the nonlinearity power $\sigma$ is positive and $\delta_0$ is a Dirac's delta potential centred at the origin of the two-dimensional space. Notice that the nonlinearity is embodied in the coupling of the delta potential, while the focusing character results from imposing $\beta > 0$. Let us point out that in \cite{CCT} both attractive and repulsive delta potentials are considered and the related term is written as $+\beta | \psi_t |^{2 \sigma} \delta_0 \psi_t$, with $\beta$ of either sign.

As it is well-known (see e.g. \cite{Al,CCF}), in eq. \eqref{eq:NLSE} the delta term cannot be considered as a perturbation, since the Laplacian cannot control a delta potential. Nevertheless, as shown in \cite{CCT}, it is possible to define a pointwise nonlinear interaction in the following way.


\subsection{From linear to nonlinear point interactions}

The issue of a rigorous definition of point interactions  arises even in the linear case (i.e., when $\sigma=0$ in \eqref{eq:NLSE}).

\noindent A standard approach \cite{Al} considers the classifications of  all the nontrivial self-adjoint extensions of the hermitian operator $-\Delta_{|C_0^\infty(\R^2\backslash\{\zev\})}$. It turns out that there exists only a one-parameter family of self adjoint extensions, denoted by $\ham$, such that: for all $\alpha\in\R$, $\ham:L^2(\R^2)\to L^2(\R^2)$ has  domain
\begin{multline}
 \label{eq-domain}
\mathrm{dom}(\ham):=\bigg\{\psi\in L^2(\R^2): 
\phi_\lambda:=\psi-qG^{\lambda}\in H^2(\R^2),\,\,q\in\C,\,\,\:\phi_\lambda(\zev)=\bigg(\alpha+\tf{\log\f{\sqrt{\lambda}}{2}+\gamma}{2\pi}\bigg)q\bigg\}
\end{multline}
and action
\begin{equation}
 \label{eq-action}
 (\ham+\lambda)\psi:=(-\Delta+\lambda)\phi_\lambda,\qquad\forall\psi\in\mathrm{dom}(\ham),
\end{equation}
where $\lambda>0$, $\gamma$ is the Euler-Mascheroni constant and $G^{\lambda}$ is the Green's function of $-\Delta+\lambda$ in dimension two, namely
\[
G_\lambda (\xv) \ = \ \tx\f 1 {2 \pi} K_0\big(\sqrt \lambda |\xv|\big),
\]
with $K_0$ the MacDonald function of order zero (see, e.g. \cite{AS}). Note that the definition of $\ham$ is independent of $\lambda$ since, if  $\phi_\lambda\in H^2(\R^2)$, then $\phi_{\lambda'}\in H^2(\R^2)$ for every other $\lambda'>0$, given that $G^\lambda-G^{\lambda'}\in H^2(\R^2)$. As a consequence, $\lambda$ is just a mute parameter connected with the several possible representations of a function of $\mathrm{dom}(\ham)$, which plays no relevant role in the following.

Self-adjointness of $\ham$ entails, via  Stone theorem, that for every $\psi_0\in\mathrm{dom}(\ham)$, there exists a unique function
\begin{equation}
 \label{eq-reglin}
 \psi\in C\big([0,T];\mathrm{dom}(\ham)\big)\cap C^1\big([0,T];L^2(\R^2)\big),\quad\forall T>0,
\end{equation}
which solves
\begin{equation}
 \label{eq:LSE}
 \left\{
 \begin{array}{l}
  \displaystyle i\partial_t \psi_t= \ham\psi_t\\[.2cm]
  \displaystyle \psi_{t=0}=\psi_0.
 \end{array}
 \right.
\end{equation}
Note that \eqref{eq-reglin}, according to the definition of $\mathrm{dom}(\ham)$, implies that at every time $t$ the solution decomposes as
\begin{equation}
 \label{eq-dec}
 \psi_t=\phi_{\lambda,t}+q(t)G^\lambda
\end{equation}
with $q(t)\in\C$, $\phi_{\lambda,t}\in H^2(\R^2)$ and
\begin{equation}
 \label{eq-bc}
 \phi_{\lambda,t}(\zev)=\bigg(\alpha+\frac{\log\f{\sqrt{\lambda}}{2}+\gamma}{2\pi}\bigg)q(t).
\end{equation}
On the other hand, it is well-known that, denoting by $U_0(t)=e^{i t \Delta}$ the free Schr\"odinger propagator, with kernel $U_0(t;|\xv|) : = \frac{e^{-\frac{|\xv|^{2}}{4 i t}}}{2 i t}$ (acting by the normalized convolution product
$\lf(f * g \ri)(\xv):=\frac{1}{2\pi}\int_{\R^2}\dyv\:f(\xv-\yv)g(\yv)$)
and by $\I$ the {Volterra function} of order $-1$ \cite{CCT}, i.e.
\[
 \I(t) : = \int_{0}^{\infty}\dtau \: \frac{t^{\tau - 1}}{\Gamma(\tau)},
\]
($\Gamma$ is the Euler function), the solution $\psi_t$ can be written as
\begin{equation}
  \label{eq:ansatz}
  \psi_t (\xv) := (U_0(t)\psi_0) (\xv)+\frac{i}{2\pi}\int_0^t \dtau\: U_0(t-\tau; |\xv|) \: q(\tau),
 \end{equation}
where $q$ is the unique solution of the so-called \emph{charge equation}
\begin{equation}
 \label{eq:charge_lin}
  q(t)+4 \pi \int_0^t\dtau\:\I(t-\tau)\bigg(\alpha+\f{\gamma-\log2}{2\pi} - \f i 8 \bigg)q(\tau)=4\pi\int_0^t\dtau\:\I(t-\tau)(U_0(\tau)\psi_0)(\zev).
 \end{equation}

Then, nonlinear point interactions arise as one assumes that the strength of the point interaction $\alpha$ depends on the wave functions. Specifically, one sets $\alpha=-\beta|q(t)|^{2\sigma}$, so that \eqref{eq:charge_lin} reads
\begin{equation}
 \label{eq:charge}
 q(t)+4 \pi \int_0^t\dtau\:\I(t-\tau)\bigg(\theta_1(|q(\tau)|) - \f i 8 \bigg)q(\tau)
 \\[.2cm]=4\pi\int_0^t\dtau\:\I(t-\tau)(U_0(\tau)\psi_0)(\zev).
\end{equation}
with
\beq\label{eq-theta}
	\theta_\lambda (s) := 
\left(\tfrac{1}{2\pi}\log\tfrac{\sqrt{\lambda}}{2}+\tfrac{\gamma}{2\pi} -  \beta s^{2 \sigma} \right),\qquad \lambda>0,\quad s\geq0.
\eeq 
\noindent Hence the evolution problem combines  \eqref{eq:charge}, whose solution is called \emph{charge}, and \eqref{eq:ansatz}, which is expected to solves the nonlinear version of \eqref{eq:LSE}, namely
\begin{equation}
 \label{eq:NLSEstrong}
 \left\{
 \begin{array}{l}
  \displaystyle i\partial_t \psi_t= \mathcal{H}_{\alpha=-\beta|q|^{2\sigma}}\psi_t\\[.2cm]
  \displaystyle \psi_{t=0}=\psi_0,
 \end{array}
 \right.
\end{equation}
where $\mathcal{H}_{\alpha=-\beta|q|^{2\sigma}}$ is the nonlinear map that arises letting $\alpha=-\beta|q|^{2\sigma}$ in \eqref{eq-domain}-\eqref{eq-action}.


\subsection{Setting of the problem and previous results}

As explained in \cite{CCT}, the proof of the well-posedness of \eqref{eq:NLSEstrong} in a strong sense is still open at the moment,  since a  proof of a sufficient regularity for the solution of \eqref{eq:charge} is lacking.

\noindent Nevertheless,  a weak version of \eqref{eq:NLSEstrong} has been proved to be well-posed in \cite{CCT}.  Indeed, the weak form of \eqref{eq:LSE} is given by
\begin{equation}
 \label{eq-LSEweak}
 \left\{
 \begin{array}{ll}
 \displaystyle i\tfrac{d}{dt}\langle\chi,\psi_t\rangle=\mathcal{F}_\alpha[\psi_t,\chi], & \forall\chi\in V,\\[.2cm] 
 \displaystyle \psi_{t=0}=\psi_0, & 
 \end{array}
\right.
\end{equation}
where $\langle\cdot,\cdot\rangle$ is the ordinary hermitian product in $L^2 (\R^2)$ and, $\mathcal{F}_\alpha$ is the sequilinear form associated with $\ham$, i.e.
\[
\mathcal{F}_{\alpha}(\psi):=\|\nabla\phi_{\lambda}\|_{L^2(\R^2)}^2+\lambda\big(\|\phi_{\lambda}\|_{L^2(\R^2)}^2-\|\psi\|_{L^2(\R^2)}^2\big)+\bigg(\alpha+\f{\log\f{\sqrt{\lambda}}{2}+\gamma}{2\pi}\bigg)|q|^2,
\]
with domain
\begin{equation}
 \label{eq:form_dom}
 V := \big\{\psi\in L^2(\R^2):\phi_\lambda:=\psi-qG^{\lambda}\in H^1(\R^2),\,\,q\in\C\big\}.
\end{equation}
Note that the form domain $V$ induces the same decomposition of $\mathrm{dom}(\ham)$ (i.e., \eqref{eq-dec}), but presents a weaker regularity requirement on $\phi_\lambda$ ($H^1(\R^2)$ in place of $H^2(\R^2)$) and, moreover, no {boundary condition} as \eqref{eq-bc} is imposed. Note also that, as well as for $\ham$, $\mathcal{F}_\alpha$ and $V$ do not depend on the parameter $\lambda$.

\noindent As a consequence, the nonlinear version of \eqref{eq-LSEweak} is obtained by setting  $\alpha=-\beta|q|^{2\sigma}$ in \eqref{eq-LSEweak}, thus yielding

\begin{equation}
 \label{eq:NLSEweak}
 \left\{
 \begin{array}{ll}
 \displaystyle i\tfrac{d}{dt}\langle\chi,\psi_t\rangle=\langle\nabla\chi_\lambda,\nabla\phi_{\lambda,t} \rangle+\lambda \langle\chi_\lambda,\phi_{\lambda,t}\rangle-\lambda \langle\chi,\psi_t\rangle+ \theta_\lambda ( |q(t)| ) q_\chi^* q(t), & \forall\chi\in V,\\[.2cm] 
 \displaystyle \psi_{t=0}=\psi_0, &
 \end{array}
 \right.
\end{equation}
\noindent Concerning \eqref{eq:NLSEweak}, \cite{CCT} shows the following facts:

\begin{enumerate}[(i)]
\item if $\sigma \geq \f 1 2$, then for every $\psi_0 \in \DE$, where
\[
 \DE := \big\{ \psi \in V :(1 + k^{\ep}) \: \phitr_\lambda(\kv) \in L^1(\R^2), \mbox{ for some } \ep > 0 \big\}
\]
($k:=|\kv|$), there exists a unique solution $\psi_t \in V$ to problem \eqref{eq:NLSEweak}, for $t \in [0,T)$  \cite[Theorem 1.1]{CCT}; furthermore, $\psi_t$ is of the form \eqref{eq:ansatz} with $q(\cdot)$ the unique solution in $C[0,T]\cap H^{1/2}(0,T)$ of \eqref{eq:charge} \cite[Propositions 2.2, 2.3 \& 2.4]{CCT};

\item for the solution $\psi_t$, the following identities hold for all $t \in [0,T)$ \cite[Theorem 1.2]{CCT}:
\begin{enumerate}[(a)]
\item conservation of mass:
 $$ M(\psi_t):=\left\| \psi_t\right\|_{L^2(\R^2)} = M (\psi_0); $$ 
\item conservation of energy:
 \bml{
  \label{eq:energy}
  E(\psi_t) :=  \, \| \nabla \phi_{\lambda,t} 
\|_{L^2 (\R^2, \C^2)}^2 + \lambda \big(\| \phi_{\lambda,t} \|_{L^2 (\R^2)}^2 - 
\| \psi_t \|_{L^2 (\R)}^2\big) +\\+ \theta_\lambda (|q(t)|)|q(t)|^2 + \f {\sigma \beta}
 {\sigma + 1} |q(t)|^{2 \sigma + 2} = E (\psi_0).}
\end{enumerate}
\end{enumerate}

\begin{rem}[Energy expression]
	\mbox{}	\\
The expression of energy  differs from the one used in \cite{CCT} in two respects: first, here it is stated for a generic $\lambda > 0$, while in \cite{CCT} it is given for $\lambda = 1$. However, this does not actually affects the value of the energy, which independent of $\lambda$. Second, for the sake of obtaining a shorthand expression in \cite{CCT} the term $M^2 (\psi_t)$ is added to the energy,
which does not affect the conservation law. In this paper, we prefer the expression \eqref{eq:energy} since it gives a more straightforward energy threshold for the blow-up.
\end{rem}


\subsection{Main results}

In the results previously recalled, nothing is said about the possibility of extending the local solution to arbitrarily large times (this is guaranteed, indeed, by \cite[Theorem 1.3]{CCT} only for the \emph{defocusing} case, i.e., $\beta<0$). In fact, we will show that this is not always the case: however small the nonlinearity power $\sigma$ is, there always exist initial data for which the solution cannot be extended beyond a certain finite time. Let us first give a basic definition:

\begin{defi}[Blow-up solutions]
\mbox{}	\\
A solution $\psi_t$ to problem \eqref{eq:NLSEweak} is said to {\em blow up in finite time} and is called a {\em blow-up solution}, if there exists $T_* > 0$ such that
$$ \limsup_{t \to T_*} | q (t) | = + \infty.$$\end{defi}

\begin{rem}[Blow-up of the regular part]
\mbox{}	\\
Note that a pointwise blow-up of the charge $ q(t) $ at $ T_{*} $ implies the explosion at the same time of the $ H^1 $ norm of the regular part $ \phi_{\lambda,t} $, due to the energy conservation \eqref{eq:energy}. On the other hand, since also the mass is preserved, whenever $ q(t) $ blows up at $ T_{*} $, the $L^2$ norm of the regular part has to blow up as well. \end{rem}

Also, recall that, according to \cite[Proposition 1.1]{CCT}, a {\em blow-up alternative} holds, namely a solution to \eqref{eq:NLSEweak} cannot be extended to a global one if and only if it blows up in finite time.

\medskip
The main result of the present paper is the following:

\begin{teo}[Sufficient condition for blow-up] \label{teo:sufficient} 
	\mbox{}	\\
	Let $\sigma \geq 1/2$ and let the initial datum 
	\begin{equation}
 		\label{eq:schwarz_domain}
 		\psi_0 \in \DES:=\lf\{\psi\in\DE:\phi_\lambda\in\SC \ri\}
	\end{equation}
	($\SC$ denoting the space of Schwartz functions) satisfy the energy condition
\begin{equation}
  \label{eq:blow_cond}
  E(\psi_0) < \Lambda:=- \frac \sigma {4 \pi (\sigma + 1)
(4 \pi \sigma \beta)^{\f 1 \sigma}}.
\end{equation}
 Then, the solution $\psi_t$ to \eqref{eq:NLSEweak} blows up in finite time.
\end{teo}

As already stressed, the most relevant difference with respect to the 1D and the 3D cases is the lack of a {critical power}, namely the existence of a minimum value of $\sigma$ in order to have blow-up solutions: for both the one and the three-dimensional case such a value equals one \cite{AT,ADFT2}. Notice that the assumption in Theorem \ref{teo:sufficient} about $\sigma$ is merely technical, as it guarantees local well-posedness, and is not related to criticality \cite[Remark 1.4]{CCT}. The restriction on the choice of the initial data, on the other hand, could be weakened; nevertheless we chose to keep it not to make computations too burdensome.

The proof of Theorem \ref{teo:sufficient} is quite immediate, provided that one knows the qualitative behavior of the so-called \emph{moment of inertia} associated with a solution $\psi_t$, i.e.
\begin{equation}
 \label{eq:inertia}
\M(t):=\int_{\R^2}\dxv\:|\xv|^2\left|\psi_t(\xv)\right|^2.
\end{equation}
The discussion of the behavior of $\M$ is the main technical problem of the present paper and will be extensively addressed in the following sections. However, once the behavior of $\M$ is understood (Corollary \ref{cor-entails}), the proof of Theorem \ref{teo:sufficient} is immediate:

\begin{proof}[Proof of Theorem \ref{teo:sufficient}]
From Corollary \ref{cor-entails},
\begin{equation}
\label{eq-mombound}
 \ddot \M (t) \leq 8 (E (\psi_0) - \Lambda),\qquad\forall t\in[0,T_*).
\end{equation}
Hence, hypothesis \eqref{eq:blow_cond} entails that the moment of inertia is uniformly concave and this would contradict the positivity of $\M$ unless $T_*<+\infty$. Then, the blow-up alternative implies that $\psi_t$ blows up in a finite time.
\end{proof}

The threshold $\Lambda$ has an interesting connection with the energy of the \emph{standing waves} of the problem, that is

\begin{defi}[Standing waves]\label{de-standing}
\mbox{}	\\
A nontrivial solution $\psi^\omega (t,x)$ to \eqref{eq:NLSEweak} of the form
\begin{equation}
\psi^\omega (t,\xv) \, = \, e^{i \omega t} u^{\omega} (\xv) \label{standing}
\end{equation}
is said a {\em standing wave} of \eqref{eq:NLSEweak}.
\end{defi}

Precisely, it is possible to completely classify the standing waves of \eqref{eq:NLSEweak} and see that the infimum of their energies equals $\Lambda$.

\begin{teo} [Standing waves]\label{teo-stand}
\mbox{}	\\
Every standing wave of \eqref{eq:NLSEweak} is given by
\begin{equation} 
u^\omega (\xv) = Q(\omega) e^{i\eta} G_\omega (\xv) \label{uomega},
\end{equation}
where $ \eta \in \R $ is a constant,
\begin{equation}
 Q (\omega) := \left( \frac{\log \frac {\sqrt \omega} 2 + \gamma}{2 \pi \beta} \right)^{\frac 1 {2 \sigma}},
\label{quomega}
\end{equation}
and $\omega \in (4 e^{-2 \gamma}, + \infty)$. In addition,
\begin{equation}
 \label{infenergy}
 \min_{\omega \in (4 e^{- 2 \gamma}, + \infty)} E (\psi^\omega )=\Lambda<0.
\end{equation}
\end{teo}

\begin{rem}[Energy threshold]
\mbox{}	\\
 Given the coincidence of the energy threshold in \eqref{eq:blow_cond} for the blow-up with the lowest energy of standing waves in \eqref{infenergy}, it is intuitive that the uncommon features of the blow-up, for the pointwise NLS in dimension two, has to be connected to uncommon features of the standing waves. Hence, a deeper analysis of the behavior of $\{u^\omega\}$ will be the subject of our future investigation (and will be presented in a forthcoming paper).
\end{rem}


\section{Standing Waves}
\label{sec:standing}

This section presents the discussion on the standing waves of \eqref{eq:NLSEweak}, introduced by Definition \ref{de-standing}. In what follows we will refer both to $\psi^\omega$ and to $u^\omega$ as to a standing wave (since it does not give rise to misunderstandings).

Preliminarily, note that, in view of \eqref{eq:form_dom} (i.e., the definition of the space $V$), given $\lambda > 0$, one has the decomposition
\begin{equation}
u^\omega = \phi^\omega_\lambda + q^\omega G_\lambda \label{standec},
\end{equation}
with $\phi^\omega_\lambda \in H^1 (\R^2)$. Then, we can show, as a first step, that the frequency $\omega$ must be positive.

\begin{pro}[Standing wave frequency]
\mbox{}	\\
Let $\psi^\omega$ be a standing wave of \eqref{eq:NLSEweak}. Then, $\omega > 0$.
\end{pro} 

\begin{proof}
By \eqref{standing}, for any $\lambda > 0$, equation \eqref{eq:NLSEweak} gives
\begin{equation} \label{standeq}
(\lambda - \omega) \langle \chi, u^\omega \rangle \ = \ 
\langle \nabla \chi_\lambda, \nabla \phi^\omega_\lambda \rangle + \lambda 
\langle \chi_\lambda,  \phi^\omega_\lambda \rangle + \theta_\lambda \lf(\lf|q^\omega\ri|\ri) q^*_\chi q^\omega.
\end{equation}
Choosing $\chi \in H^1 (\R^2)$ one has $\chi = \chi_\lambda$ and $q_\chi = 0$, so that, using \eqref{standec},
$$
 - \omega \langle \chi, \phi_\lambda^\omega \rangle +
(\lambda - \omega) q^\omega \langle \chi, G_\lambda \rangle
\ = \ 
\langle \nabla \chi, \nabla \phi^\omega_\lambda \rangle. 
$$
Then, by density of $H^1 (\R^2)$ in $L^2 (\R^2)$, 
$$ - \Delta \phi^\omega_\lambda \ =  - \omega \phi^\omega_\lambda + (\lambda - \omega)
q^\omega G_\lambda,
$$
and, finally, in the Fourier space one gets
\[
(k^2 + \omega) \widehat \phi^\omega_\lambda \ = \ \frac {(\lambda - \omega) q^{\omega} }
{2 \pi (k^2 + \lambda)}.
\]
If $\omega > 0$, then
$$
\widehat \phi^\omega_\lambda \ = \ \frac { (\lambda - \omega) q^\omega}
{2 \pi (k^2 + \lambda)(k^2 + \omega)}
$$
so that $\phi^\omega_\lambda \in H^1 (\R^2)$. If, conversely, $\omega \leq 0$, then $ \widehat \phi^\omega_\lambda \in L^2(\R^2) $ if and only if $ q^{\omega} = 0 $, and thus $ \phi^\omega_\lambda  = 0 $, which implies that $ u^{\omega} $ cannot be a standing wave.
\end{proof}

Exploiting the previous result, we can now prove Theorem \ref{teo-stand}.

\begin{proof}[Proof of Theorem \ref{teo-stand}]
Since $\omega > 0$, one can choose $\lambda = \omega$ in \eqref{standec}, so that \eqref{standeq} gives
\begin{equation} \label{standeqpos}
0 \ = \ 
\langle \nabla \chi_\omega, \nabla \phi^\omega_\omega \rangle + \omega 
\langle \chi_\omega,  \phi^\omega_\omega \rangle + \theta_\omega \lf(\lf|q^\omega\ri|\ri)  q_\chi^* q^\omega.
\end{equation}
First, choose $q_\chi = 0$, so that $\chi = \chi_\omega$. Thus
$$ 0 \ = \ \langle \chi, ( - \Delta + \omega ) \phi_\omega^\omega \rangle, $$
for all $\chi \in H^1 (\R^2)$, and hence $\phi^\omega_\omega = 0$.

On the other hand, choose $\chi = G_\omega$. As a consequence 
$\chi_\omega = 0$ and $q_\chi = 1$, which entails, from \eqref{standeqpos}, that
\beq
	\label{condition sw}
\theta_\omega \lf(\lf|q^\omega\ri|\ri)  q^\omega = 0.
\eeq
Then either $q^\omega = 0$ or $\theta_\omega \lf(\lf|q^\omega\ri|\ri) = 0$. In the first case, 
we have $u^\omega = 0$, so it is not a standing wave.
In the second case, $ \theta_\omega \lf(\lf|q^\omega\ri|\ri) = 0$ implies that $ \lf| q^{\omega} \ri| $ equals the r.h.s. of \eqref{quomega}. In addition, such a quantity must be positive, which implies the condition $\omega > 4 e^{-2 \gamma}$ and then \eqref{uomega}.

Finally, by direct computation,
$$ E (\psi^\omega_t) \ = \ - \f {|q^{\omega}|^2}{4 \pi} + \f {\sigma \beta}{\sigma + 1} |q^{\omega}|^{2 \sigma + 2}, $$
thus, minimizing in $|q^{\omega}| \in (0, + \infty)$, one gets \eqref{infenergy}.
\end{proof}

\begin{rem}[Parametrization of standing waves]
	\mbox{}	\\
	Notice that by the condition \eqref{condition sw} there is a one-to-one correspondence between the absolute value of the charge, i.e. $|q^\omega| \in (0, + \infty)$, and the frequency $\omega \in (4e^{-2 \gamma}, + \infty)$ of any standing wave. Therefore, the standing waves can be equivalently parametrized by $\omega$ or by $|q|$ (where the dependence on $ \omega $ can be dropped).
\end{rem}


\section{Moment of Inertia}
\label{sec:prel} 

As we mentioned in the Introduction, the main technical point of the paper is the analysis of the moment of inertia $ \M(t) $ associated with the solution $\psi_t$, defined by \eqref{eq:inertia}.

Before presenting the formulae for the time derivatives of $ \M(t) $, let us stress some basic facts, which will be useful in the following. First, we recall that, using the Fourier transform  in \eqref{eq:ansatz}, one gets
\begin{equation}
 \label{eq:psi_stand}
 \psitr_t(\kv)=e^{-ik^2t}\phitr_{\lambda,0}(\kv)+\frac{e^{-ik^2\tau} q(0)}{2\pi(k^2+\lambda)}+\frac{i}{2\pi}\int_0^t\dtau\:e^{-ik^2(t-\tau)}q(\tau),
\end{equation}
and then, integrating by parts as in \cite[Eqs. (2.4)-(2.6)]{ADFT2} one finds
the decomposition
 \begin{equation}
  \label{eq:psi_tr}
  \psitr_t(\kv)=e^{-ik^2t}\phitr_{\lambda,0}(\kv)+\frac{q(t)}{2\pi(k^2+\lambda)}+\ftr_{1,\lambda}(t,\kv)+\ftr_{2,\lambda}(t,\kv),
 \end{equation}
 where $\ftr_{1,\lambda}$ and $\ftr_{2,\lambda}$ are given by
 \begin{gather*}
   \ftr_{1,\lambda}(t,\kv):=\frac{i\lambda}{2\pi(k^2+\lambda)}\int_0^t\dtau\:e^{-ik^2(t-\tau)}q(\tau),\\[.2cm]
   \ftr_{2,\lambda}(t,\kv):=\frac{-\lambda}{2\pi(k^2+\lambda)}\int_0^t\dtau\:e^{-ik^2(t-\tau)}\dot{q}(\tau).
 \end{gather*}
 Note that the well-definition of the last integral is a straightforward consequence of Proposition \ref{pro:reg_q}, where we will prove absolute continuity of the charge. Furthermore, by the same computations leading to \cite[Eq. (2.54)]{CCT} and following, one has
that $f_{1, \lambda}(t)$ and $f_{2, \lambda} (t)$ belong to $H^1 (\R^2)$ and that
\begin{equation}
\label{effeunodue}
\lf\| f_{j, \lambda} \ri\|_{L^\infty ((0,T), H^1 (\R^2))} \leq C, \qquad j = 1,2, \ \forall T \in (0, T_*),
\end{equation}
where the constant $C$ depends (possibly) on $T$.

On the other hand, by direct computation,
\[
  \nabk\ftr_{1,\lambda}(t,\kv)=\frac{-i\lambda\kv}{\pi(k^2+\lambda)^2}\int_0^t\dtau\:e^{-ik^2(t-\tau)}q(\tau)+\frac{\lambda\kv}{\pi(k^2+\lambda)}\int_0^t\dtau\:e^{-ik^2(t-\tau)}(t-\tau)q(\tau),
 \]
 and
 \begin{equation}
  \label{eq:gradfdue}
  \nabk\ftr_{2,\lambda}(t,\kv)=\frac{\kv}{\pi(k^2+\lambda)^2}\int_0^t\dtau\:e^{-ik^2(t-\tau)}\dot{q}(\tau)+\frac{i\kv}{\pi(k^2+\lambda)}\int_0^t\dtau\:e^{-ik^2(t-\tau)}(t-\tau)\dot q(\tau),
 \end{equation}
and, since the action of $\nabk$ is essentially a multiplication by $k$, \eqref{effeunodue} translates into 
\begin{equation}
\label{gradeffeunodue}
\lf\| \nabk f_{j, \lambda}  \ri\|_{L^\infty ((0,T), L^2 (\R^2))} \leq C, \qquad j = 1,2, \ \forall T \in (0, T_*),
\end{equation}
and
\begin{equation}
\label{kgradeffeunodue}
\lf\| {\bf k} \cdot \nabk f_{j, \lambda} \ri\|_{L^\infty ((0,T), H^{-1} (\R^2))} \leq C, \qquad j = 1,2, \ \forall T \in (0, T_*).
\end{equation}

\medskip
In view of the previous remarks, we can prove that the moment of inertia defined by \eqref{eq:inertia} is finite for all $t \in [0, T_*)$.

\begin{lem}\label{lem-weldef}
\mbox{}	\\
 Let $\sigma\geq1/2$ and let $\psi_0\in\DES$ (with $\DES$ defined by \eqref{eq:schwarz_domain}). Then, $ \M \in L^\infty(0,T)$ for any $T<T_*$.
\end{lem}

\begin{proof}
 From \eqref{eq:psi_tr},
 \begin{equation}
  \label{eq:gradpsi}
  \nabk\psitr_t(\kv)=-2it\kv e^{-ik^2t}\phitr_{\lambda,0}(\kv)+e^{-ik^2t}\nabk\phitr_{\lambda,0}(\kv)-\frac{q(t)\kv}{\pi(k^2+\lambda)^2}+\nabk\ftr_{1,\lambda}(t,\kv)+\nabk\ftr_{2,\lambda}(t,\kv).
 \end{equation}
 Since
\[
  \M(t)=\int_{\R^2}\dkv\:|\nabk\psitr_t(\kv)|^2,
 \]
one has to prove that the $L^2$-norm of the terms in \eqref{eq:gradpsi} is bounded in $[0,T]$, for every fixed $T<T_*$. For the first three terms this is immediate,
 while for the last two terms it follows from \eqref{gradeffeunodue}.
\end{proof}


\subsection{Derivatives of the moment of inertia}

Now, we can present the main technical results of the paper, that is the formulae for the first and the second derivative of the moment of inertia. In this section, we only mention the statements of the results and show some heuristic computations, in order to give an intuitive idea of the reasons for which one should expect these formulae. The rigorous proofs are postponed to Section \ref{app-der}. Notice that the results presented below hold true also in the defocusing case, i.e., if $ \beta < 0 $.

\begin{pro}[First derivative of $ \M $]
 \label{pro:first_der}
 \mbox{}	\\
 Let the assumptions of Lemma \ref{lem-weldef} be satisfied. Then, $ \M \in AC[0,T]$ for any $ T < T_* $ and 
\beq
	\label{derivative M}
  \dot{\M}(t)=4\:\PI\bigg\{\int_{\R^2}\dkv\:\psitr_t(\kv)\:\kv\cdot\nabk \psitrs(\kv)\bigg\},\qquad\text{for a.e.} \: t\in[0,T_*).
 \eeq
\end{pro}

Formula \eqref{derivative M} is quite classical in the theory of blow-up also for standard NLS equations, and goes under the name called virial identity. Its formal derivation goes as follows: neglecting any regularity issues
 \[
\dot\M(t)=\frac{\diff}{\diff t}\int_{\R^2}\dxv\:|\xv|^2\left|\psi_t(\xv)\right|^2=\frac{\diff}{\diff t}\int_{\R^2}\dkv\:|\nabk\psitr_t(\kv)|^2=2\PR\bigg\{\int_{\R^2}\dkv\:\partial_{t}\left(\nabk\psitr_t(\kv)\right)\nabk\psitr_t^{*}(\kv)\bigg\}.
\]
Now, since by \eqref{eq:gradpsi},
\beq
	\label{eq: cazzo}
\partial_{t}\left(\nabk\psitr_t(\kv)\right)=-2i\kv\psitr_t(\kv)-i |\kv|^{2}\nabk\psitr_t(\kv),
\eeq
then one gets
\begin{multline*}
\dot\M(t)=2\PR\bigg\{\int_{\R^2}\dkv\:\left(-2i\kv\psitr_t(\kv)\nabk\psitr_t(\kv)^{*}-i |\kv|^{2}|\nabk\psitr_t(\kv)|^{2}\right)\bigg\}\\[.2cm]=4\PI\bigg\{\int_{\R^2}\dkv\:\kv\psitr_t(\kv)\nabk\psitr_t^{*}(\kv)\bigg\}.
\end{multline*}

\medskip
On the other hand, exploiting \eqref{derivative M}, it is possible to establish the following formula, which is the central point of Glassey method.

\begin{pro}[Second derivative of $ \M $]
 \label{pro:second_der}
 \mbox{}	\\
 Let the assumptions of Lemma \ref{lem-weldef} be satisfied. Then, $ \M(t) \in C^2[0,T] $ for any $ T < T_* $ and 
 \begin{equation}
  \label{eq:second_der}
  \ddot{\M}(t)=8E(\psi_0)+2\left(\frac{1}{\pi}-\frac{4\beta\sigma}{\sigma+1}|q(t)|^{2\sigma}\right)|q(t)|^2,\qquad\forall\,t\in[0,T_*).
 \end{equation}
\end{pro}

Proposition \ref{pro:second_der} has an immediate consequence, which is the main tool used in the proof of Theorem \ref{teo:sufficient}.

\begin{cor}[Threshold $ \Lambda $ and $ \ddot \M $]\label{cor-entails}
\mbox{}	\\
Under the assumptions of Lemma \ref{lem-weldef},
\begin{equation} \label{entails}
\ddot \M (t) \leq 8 \lf(E (\psi_0) - \Lambda \ri),\qquad\forall\, t\in[0,T_*),
\end{equation}
where $\Lambda$ is defined by \eqref{eq:blow_cond}.
\end{cor}

\begin{proof}
Notice that, given a solution $\psi_t$ and denoted by $u^{\omega_t}$ the unique positive
standing wave whose charge equals $|q(t)|$, the identity \eqref{eq:second_der} rewrites as
\[
\ddot \M (t) = 8 \lf( E (\psi_0) - E (u^{\omega_t}) \ri).
\]
Hence, minimizing $E (u^{\omega_t})$ in $|q(t)|$, as in the proof of Theorem \ref{teo-stand}, \eqref{entails} follows.
\end{proof}

\begin{rem}[Concavity of $ \M $ and critical exponent]
	\mbox{}	\\
 In \eqref{eq:second_der} one can see the technical reason for which in the 2D case the problem of the blow-up does not present any critical exponent for the nonlinearity, unlike in the 1D and 3D cases. Indeed, in view of \eqref{eq:second_der}, in order to impose the uniform concavity of $\M$ the exponent $\sigma$ plays no significant role. In other words, for any $\sigma$ ($\geq1/2$) there exists a sufficient condition for the blow-up. On the contrary, in the 1 or 3D cases the second derivative of the moment of inertia reads \cite{ADFT1,AT}
 \[
  \ddot{\M}(t)=8E(\psi_0)-4\beta\frac{\sigma-1}{\sigma+1}|q(t)|^{2\sigma+2}
 \]
and thus the role of the exponent $\sigma=1$ is apparent. In addition, in those cases it is possible to prove that when $\sigma<1$ the solution is global.
\end{rem}

As for the first derivative, let us show some heuristic derivation of \eqref{eq:second_der}: we assume here for simplicity that $ \psi_t $ is a strong solution of the Cauchy problem \eqref{eq:NLSEweak}, i.e. at any time $ t \in \R $, $ \psi_t $ belongs to the domain of the nonlinear operator appearing on the r.h.s. of \eqref{eq:NLSEweak}. This simply implies \cite{CCF} that $ \psi_t $ admits the usual decomposition $ \psi_t = \phi_{\lambda,t} + q(t) G_{\lambda} $, with $ \phi_{\lambda,t} \in H^2(\R^2) $ and $ q(t) \in \C $ satisfying the {\it boundary condition}
\beq
	\label{eq: bc}
	\phi_{\lambda,t}(\mathbf{0}) = \theta_{\lambda}(|q(t)|) q(t).
\eeq
Under this assumption, using \eqref{eq-dec}, with $\lambda=1$ for the sake of simplicity, \eqref{eq: cazzo} and Divergence Theorem and differentiating \eqref{eq:psi_smooth} (see also \eqref{eq:Ie_aux}), we get
\bml{\label{fII0}
\ddot{\M}(t)=4\:\PI\bigg\{\int_{\R^2}\dkv\:\lf[\partial_t\psitr_t(\kv)+ik^2\psitr_t(\kv)\ri]\kv\cdot\nabk\psitrs(\kv)\bigg\}+8\int_{\R^2}\dkv\:\,k^2\big|\psitr_t(\kv)\big|^2\\
=\frac{2}{\pi}\:\PR\left\{q(t)\int_{\R^2}\dkv\:\kv\cdot\nabk \psitrs(\kv)\right\}+8\int_{\R^2}\dkv\:\,k^2\big|\psitr_t(\kv)\big|^2 \\
= \frac{2}{\pi}\:\PR\left\{q(t)\int_{\R^2}\dkv\:\kv\cdot\nabk \phitr^*_{\lambda,t}(\kv)\right\}+8\int_{\R^2}\dkv\:\,k^2 \lf[ \big|\phitr_{\lambda,t}(\kv) \big|^2 + \frac{1}{\pi} \PR  \bigg\{\frac{q(t)\phitr^*_{\lambda,t}(\kv)}{k^2 + \lambda} \bigg\}\ri] \\
= 8 \lf( \PR \big\{ q(t) \phi_{\lambda,t}^*(\mathbf{0}) \big\} + \lf\| \nabla \phi_{\lambda,t} \ri\|_2^2 - 2 \lambda \PR \big\{\langle\phi_{\lambda,t},q(t) G_{\lambda}\rangle\big\} \ri),
}
where we neglected again any regularity issue. Plugging in the above expression the value of $ \phi_{\lambda,t}(\mathbf{0}) $ given by the condition \eqref{eq: bc} and recalling the expression \eqref{eq:energy} of the energy $ E(\psi_t) = E(\psi_0) $, \eqref{eq:second_der} is recovered.


\section{Proofs of Propositions \ref{pro:first_der} and \ref{pro:second_der}}
\label{app}

This section is completely devoted to the rigorous proof of Propositions \ref{pro:first_der} and \ref{pro:second_der}, which make then rigorous the formal computations presented before.

\subsection{Extra-regularity of the charge}
\label{app-charge}

The first step to prove \eqref{derivative M} and \eqref{eq:second_der} is that of establishing some further regularity for the charge $q(\cdot)$, with respect to the one obtained by \cite{CCT} (namely, $H^{1/2}(0,T)\cap C[0,T]$\footnote{Actually, in \cite[Lemma 2.6]{CCT} is proved the \emph{log-H\"older continuity} of the charge, but it is not sufficient as well.}), possibly exploiting the more restrictive assumptions on the smoothness of initial data.

As we will see in the following, the required property is the absolute continuity on closed and bounded intervals. The proof of such regularity follows exactly the strategy developed by \cite{CCT} in order to prove $H^{1/2}$-regularity (precisely, \cite[Proposition 2.3 \& 2.4]{CCT}). As a consequence, we discuss here only new technical aspects, referring to \cite{CCT} for those results which do not require significant modifications.

The first step is to establish Lipschitz continuity of the map $f\mapsto|f|^{2\sigma}f$ in $W^{1,1}(0,T)$ (which is the analogue of \cite[Lemma 2.1]{CCT}).

\begin{lem}\label{lem-lip}
\mbox{}	\\
 Let $\sigma\geq\frac{1}{2}$ and $T,\,M > 0$. Assume also that $f$ and $g$ are functions satisfying
 \begin{equation}
  \label{eq-lipass}
  \left\|f\right\|_{L^\infty(0,T)}+ \left\|f\right\|_{W^{1,1}(0,T)} \leq M, \qquad\left\|g\right\|_{L^\infty(0,T)}+ \left\|g\right\|_{W^{1,1}(0,T)}\leq M.
 \end{equation}
 Then, there exists a constant $C>0$ independent of $f,\,g,\,M\mbox{ and }T$, such that
 \begin{equation}
  \label{eq:lip_sob}
  \left\||f|^{2\sigma}f-|g|^{2\sigma}g\right\|_{W^{1,1}(0,T)} \leq C\max\left\{1,T\right\} M^{2\sigma}\left(\left\| f-g \right\|_{L^\infty(0,T)} + \left\| f-g \right\|_{W^{1,1}(0,T)}\right).
 \end{equation}
\end{lem}

\begin{proof}
 Denote by $\varphi:\C\to\C$ the function $\varphi(z)=|z|^{2\sigma}z$, which belongs to $C^2(\R^2;\C)$  as a function of the real and imaginary parts of $ z $ since $\sigma\geq\frac{1}{2}$. Arguing as in \cite[Proof of Lemma 2.1]{CCT}, easy computations yield
 \begin{equation}
  \label{eq-diff}
  \varphi(f(t))-\varphi(g(t))=(f(t)-g(t))\,\phi_1(t)+(f(t)-g(t))^*\,\phi_2(t),
 \end{equation}
 where $\phi_j(t):=\psi_j(f(t),g(t))$, $j=1,2$, and
 \[
  \psi_1(z_1,z_2):=\int_0^1\ds\:\partial_z\varphi(z_1+s(z_2-z_1)),\qquad \psi_2(z_1,z_2):=\int_0^1\ds\:\partial_{z^*}\varphi(z_1+s(z_2-z_1)).
 \]
 Note also that $\psi_j\in C^1(\R^4;\C)$ (now as a function of the real and imaginary parts of $ z_1 $ and $z_2$). As a consequence, one can see that
 \begin{multline*}
  \left\|\varphi(f(t))-\varphi(g(t)) \right\|_{W^{1,1}(0,T)} \leq \left\| \phi_1  \cdot \left( f - g \right) \right\|_{W^{1,1}(0,T)} + \left\| \phi_2  \cdot \left( f - g \right)  \right\|_{W^{1,1}(0,T)}\\[.2cm]
  \leq C \max\left\{\left\| \phi_1  \right\|_{L^{\infty}(0,T)} + \left\| \phi_2  \right\|_{L^{\infty}(0,T)}, \left\| \phi_1  \right\|_{W^{1,1}(0,T)} + \left\| \phi_2  \right\|_{W^{1,1}(0,T)}\right\} \times\\[.2cm]
  \times\left(\left\| f-g \right\|_{L^\infty(0,T)}+\left\| f-g \right\|_{W^{1,1}(0,T)}\right).
 \end{multline*}
 Now, combining \eqref{eq-lipass}, the definition of $\phi_j$ and the regularity assumptions on $\varphi$ (as in \cite[Proof of Lemma 2.1]{CCT}), one finds
 \begin{equation}
  \label{eq-linf}
  \left\|\phi_j\right\|_{L^\infty(0,T)}\leq C M^{2\sigma},\quad j=1,2,
 \end{equation}
 so that it is left to estimate $\left\| \phi_j  \right\|_{W^{1,1}(0,T)}$ (note that $\phi_j\in W^{1,1}(0,T)$ since it is a composition of the absolute continuous functions $f,\,g$ and $\psi_j(\cdot,\cdot)$ which is of class $C^1$). It is immediate that
 \[
  \left\|\phi_j\right\|_{L^1(0,T)}\leq CT M^{2\sigma},\quad j=1,2.
 \]
 On the other hand, from \cite[Eqs. (2.11)-(2.14)]{CCT}, one has that for a.e. $s,t\in[0,T]$
 \[
  \left|\frac{\phi_j(t)-\phi_j(s)}{t-s}\right|\leq C\max\left\{|f(t)|,|f(s)|,|g(t)|,|g(s)|\right\}^{2\sigma-1}\left(\left|\frac{f(t)-f(s)}{t-s}\right|+\left|\frac{g(t)-g(s)}{t-s}\right|\right),
 \]
 which entails
 \[
  \big|\dot{\phi}_j(t)\big|\leq CM^{2\sigma-1}\left(\big|\dot{f}(t)\big|+\big|\dot{g}(t)\big|\right),\qquad\text{for a.e.}\quad t\in[0,T].
 \]
Then,
 \[
  \big\|\dot{\phi}_j\big\|_{L^1(0,T)} \leq CM^{2\sigma-1}\left(\big\|\dot{f}\big\|_{L^1(0,T)}+\big\|\dot{g}\big\|_{L^1(0,T)}\right)\leq CM^{2\sigma}.
 \]
 Summing up, one easily obtains \eqref{eq:lip_sob}.
\end{proof}

The second step is to show that the action of a translated Volterra function preserves, as integral kernel, $W^{1,1}$-regularity. More precisely, we have

\begin{lem}
 \label{lem:contr_further}
 \mbox{}	\\
 Let $T>0$ and $h\in W^{1,1}(0,T)$. Then
 \[
  \h(t):=\int_0^T\dtau\:\I(t+T-\tau)h(\tau)
 \]
 belongs to $W^{1,1}(0,\T)$ for any $\T>0$.
\end{lem}

\begin{proof}
 One can easily see that $\h\in L^1(0,\T)$ for all $\T>0$. On the other hand, observing that
 \[
  \h(t)=\int_t^{t+T}\dtau\:\I(\tau)h(t+T-\tau),
 \]
 there results
 \[
  \dot{h}_T(t)=\I(t+T)h(0)-\I(t)h(T)+\int_0^T\dtau\:\I(t+T-\tau)\dot{h}(\tau).
 \]
 Since the first two terms are in $L^1(0,\T)$ for any $\T>0$, as functions of $t$, we must show that
 \[
  A:=\int_0^{\T}\dt\int_0^T\dtau\:\I(t+T-\tau)\big|\dot{h}(\tau)\big|<+\infty.
 \]
 Using Tonelli theorem and a change of variable, and denoting by $\NN(t):=\int_0^t\dtau\:\I(\tau)$ (which is an increasing and absolutely continuous function, as explained in \cite{CCT,CF}), one finds
 \begin{align*}
  A =  & \, \int_0^T\dtau\:\big|\dot{h}(\tau)\big|\int_{T-\tau}^{\T+T-\tau}\ds\:\I(s)= \, \int_0^T\dtau\:\big|\dot{h}(\tau)\big|\big(\NN(\T+T-\tau)-\NN(T-\tau)\big)\\[.3cm]
  \leq & \, \int_0^T\dtau\:\big|\dot{h}(\tau)\big|\NN(\T+T-\tau)\leq 2\NN(\T+T) \big\|\dot{h}\big\|_{L^1(0,T)}<\infty,
 \end{align*}
 which concludes the proof.
\end{proof}

Finally, we have all the ingredients to prove that the solution $q$ of the charge equation \eqref{eq:charge} belongs to $W^{1,1}(0,T)$ for every $T\in(0,T_*)$, where we recall that $ T_* $ is the maximal existence time. To this aim it is convenient to write \eqref{eq:charge} in the following compact form:
\begin{equation}
 \label{eq:charge_compact}
 q(t) + \int_0^t \dtau \: \bigg(g(t,\tau,q(\tau))+\kappa\I(t-\tau) \: q(\tau)\bigg) = f(t),
\end{equation}
where $\kappa:=-2\big(\log 2-\gamma+i\frac{\pi}{4}\big)$ and $g$ and $f$ are defined respectively by
\[
 g(t,\tau,q(\tau)):=-4\pi\beta\I(t-\tau)|q(\tau)|^{2\sigma}q(\tau),\qquad f(t):=4\pi\int_0^t\dtau\:\I(t-\tau)(U_0(\tau)\psi_0)(\zev).
\]
Moreover, we introduce the notation
\[
 \lf(Ig\ri)(t):=\int_0^t\dtau\:\I(t-\tau)g(\tau),\qquad t\geq0,
\]
and recall that, from\footnote{The result is actually proven there only for real functions but the extension to complex ones is trivial.} \cite[Theorem 5.3]{CF}, if $g\in W^{1,1}(0,T)$, then
\begin{equation}
 \label{eq:contrW}
 \lf\|Ig\ri\|_{W^{1,1}(0,T)}\leq \NN(T) \lf(|g(0)|+\lf\|g\ri\|_{W^{1,1}(0,T)}\ri).
\end{equation}
Note that the following result extends straightforwardly to the defocusing case $ \beta < 0 $.

\begin{pro}[$W^{1,1}$-regularity of $ q $]
 \label{pro:reg_q}
 \mbox{}	\\
 Under the assumptions of Lemma \ref{lem-weldef}, the solution of \eqref{eq:charge_compact} $ q\in W^{1,1}(0,T) $ for any $ T < T_* $.
\end{pro}

\begin{rem}[Initial datum]
\mbox{}	\\
 In fact, an inspection of the proof of Proposition \ref{pro:reg_q} below reveals that it suffices to assume $\phi_\lambda\in H^2(\R^2)$, in place of $\phi_\lambda\in\SC$. In addition, one can see that the assumption $\phi_\lambda\in\SC$ (as well as $\phi_\lambda\in H^2(\R^2)$) simplifies the proofs of \cite[Theorems 1.1, 1.2 \& 1.3]{CCT}. Indeed, in \cite{CCT} a delicate duality pairing argument is used in order to give some meaning to the formal integration of $\dot{q}$. On the contrary, in view of Proposition \ref{pro:reg_q}, a suitable regularity of the initial datum $\psi_0$ entails that all the required integrations of $\dot{q}$ are well defined in the classical Lebesgue sense (which makes all the computations easier).
\end{rem}

\begin{proof}
 We split the proof in three steps. We first prove that the forcing term enjoys the $W^{1,1}$-regularity; then we show that the regularity holds true for $ q $ as well on short intervals, via a contraction argument; finally, by gluing together solutions on different time intervals, we prove that $ q $ has the $W^{1,1}$-regularity up to the maximal existence time.
 
 \emph{Step (i).} The first point consists of proving that $f\in W^{1,1}(0,T)$ for all $T>0$. We start by fixing arbitrarily $T>0$ and recalling that, from the decomposition of the initial datum $\psi_0$,
 \[
   4\pi(U_0(\tau)\psi_0)(\zev)=\underbrace{4\pi \left(U_0(\tau)\phi_{\lambda,0} \right)(\zev)}_{=:A_1(\tau)} + \underbrace{2q(0) \left(U_0(\tau)K_0\left(\sqrt{\lambda}|\cdot|\right) \right)(\zev)}_{=:A_2(\tau)}.
 \]
 Exploiting the Fourier transform and arguing as in \cite[Proof of Proposition 2.3]{CCT} one can see that
 \[
  \|A_1\|_{H^\nu(\R)}^2\leq C\int_{\R^2}\dkv\:\big(1+k^4\big)^\nu\big|\phitr_{\lambda,0}(\kv)\big|^2.
 \]
Hence, as $\phi_{\lambda,0}\in\SC$, $A_1$ belongs to $W^{1,1}(0,T)$, so that, by \eqref{eq:contrW}, $IA_1\in W^{1,1}(0,T)$. On the other hand, from \eqref{eq-fond} and \eqref{eq-q} below,
 \[
  A_2(\tau)=q(0) \underbrace{e^{i\lambda\tau}(-\gamma-\log(\tau))}_{:=A_{2,1}(\tau)}+ q(0) \underbrace{\tfrac{e^{i\lambda\tau}}{\pi}\left(-\pi\log\lambda+Q(\lambda;\tau)\right)}_{:=A_{2,2}(\tau)}.
 \]
 Now, since $Q$ is a smooth function of $\tau$ (see, e.g., \cite{AS} for more details), again by \eqref{eq:contrW}, there results $IA_{2,2}\in W^{1,1}(0,T)$. Finally, using the property that $\I$ is a \emph{Sonine kernel} with complement $(-\gamma-\log(\tau))$, namely that
 \[
  \int_0^t\dtau\:\I(t-\tau)(-\gamma-\log(\tau))=\int_0^t\dtau\:\I(\tau)(-\gamma-\log(t-\tau))=1,\qquad\forall t\geq0,
 \]
 (see \cite[Lemma 32.1]{SKM} and \cite[Eq. (2.29)]{CCT}), we have
 \[
  \lf( IA_{2,1} \ri)(t)=1+\int_0^t\dtau\:\I(t-\tau)a_{2,1}(\tau),\qquad a_{2,1}(\tau) := \left(e^{i\lambda\tau}-1 \right)(-\gamma-\log(\tau)).
 \]
 Since $ a_{2,1} $ is  in $W^{1,1}(0,T)$, this entails that $IA_{2,1}\in W^{1,1}(0,T)$  too, so that $f\in W^{1,1}(0,T)$.

 \emph{Step (ii).}  Here we prove that the map
 \[
  \mathcal{G}(q)[t]:=f(t)-\int_0^t\dtau\:\bigg(g(t,\tau,q(\tau))+\kappa\I(t-\tau)q(\tau)\bigg)
 \]
 is a contraction in a suitable subset of $W^{1,1}(0,T)$, for a sufficiently small $T\in(0,T^*)$, which immediately implies that the unique solution of \eqref{eq:charge_compact} is of class $W^{1,1}$ at least on small intervals.
 
 Consider the set
 \[
  \B_{T} := \left\{q\in W^{1,1}(0,T) \: \Big| \: \left\|q\right\|_{L^\infty(0,T)}+\left\|q\right\|_{W^{1,1}(0,T)}\leq \mathrm{b}_{T} \right\},
 \]
 with $\mathrm{b}_{T}=2\max\{\|f\|_{L^\infty(0,T)}+\|f\|_{W^{1,1}(0,T)},1\}$. It is a complete metric space with the norm
 \[
  \left\| \cdot \right\|_{\B_{T}}= \left\|\cdot \right\|_{L^\infty(0,T)}+ \left\|\cdot \right\|_{W^{1,1}(0,T)}.
 \]
 In order to prove that $ \mathcal{G} $ is a contraction on $ \B_{T} $, we first show that $\mathcal{G}(\B_{T})\subset \B_{T}$. To this aim, split the homogenous part of $\mathcal{G}(q)[t]$ in two terms:
 \[
  \mathcal{G}_1(q)[t]:=\int_0^t\dtau\:g(t,\tau,q(\tau)),\qquad \mathcal{G}_2(q)[t]:=\kappa\int_0^t\dtau\:\I(t-\tau)q(\tau).
 \]
 From \eqref{eq:contrW}, \eqref{eq:lip_sob} and \cite[Eq. (2.5)]{CCT} (i.e., the Lipschitz continuity of the map $f\mapsto|f|^{2\sigma}f$ in $L^\infty(0,T)$), one finds
 \[
  \left\|\mathcal{G}_1(q) \right\|_{W^{1,1}(0,T)} \leq C_T \left\| |q|^{2\sigma}q \right\|_{\B_{T}} \leq C_T \mathrm{b}_{T}^{2\sigma} \left\| q \right\|_{\B_T} \leq C_T  \mathrm{b}_{T}^{2\sigma+1},
 \]
 where, from now on, $C_T$ stands for a generic positive constant such that $C_T \to 0$, as $ T \to 0 $, and which may vary from line to line. In addition, arguing as in \cite[Proof of Proposition 2.3]{CCT}, that is combining \cite[Eq. (2.20)]{CCT} (i.e., the contractive property of the operator $I$ in $L^\infty(0,T)$) and again \cite[Eq. (2.5)]{CCT}, there results
 \[
  \left\|\mathcal{G}_1(q)\right\|_{L^\infty(0,T)} \leq  C_T \mathrm{b}_{T}^{2\sigma}\left\|q \right\|_{L^\infty(0,T)} \leq C_T\mathrm{b}_{T}^{2\sigma+1}
 \]
 and thus
 \[
  \left\|\mathcal{G}_1(q)\right\|_{\B_{T}}\leq C_T\mathrm{b}_{T}^{2\sigma+1}.
 \]
 On the other hand, one can easily find that $ \left\|\mathcal{G}_2(q) \right\|_{\B_T}\leq C_T \left\|q \right\|_{\B_{T}}\leq C_T \mathrm{b}_{T}$, so that
 \[
  \left\|\mathcal{G}(q) \right\|_{\B_{T}} \leq\mathrm{b}_{T} \left[  \frac{1}{2}+ C_T \left(1+\mathrm{b}_{T}^{2\sigma}\right)\right].
 \]
 Consequently, as the term in brackets is equal to $\frac{1}{2}+o(1)$ as $T\to0$, for $T$ sufficiently small, $\mathcal{G}(q)\in \B_{T}$.
 
 Therefore, it is left to prove that $\mathcal{G}$ is actually a contraction. Given two functions $q_1,\,q_2\in \B_{T}$, we have
 \[
  \mathcal{G}(q_1)-\mathcal{G}(q_2)=\mathcal{G}_1(q_1)-\mathcal{G}_1(q_2)+\mathcal{G}_2(q_1-q_2).
 \]
 Arguing as before, one sees that $ \left\|\mathcal{G}_2(q_1-q_2) \right\|_{\B_{T}}\leq C_T\left\|q_1-q_2 \right\|_{\B_{T}}$, while, using once again \eqref{eq:lip_sob}, \cite[Eqs. (2.5)$\&$(2.20)]{CCT} and \eqref{eq:contrW},
 \[
  \left\|I\left(|q_1|^{2\sigma}q_1-|q_2|^{2\sigma}q_2\right)\right\|_{\B_{T}} \leq C_T \left\||q_1|^{2\sigma}q_1-|q_2|^{2\sigma}q_2\right\|_{\B_{T}} \leq C_T \mathrm{b}_{T}^{2\sigma} \left\|q_1-q_2 \right\|_{\B_{T}}.
\]
 Then,
 \[
  	\left\|\mathcal{G}(q_1)-\mathcal{G}(q_2) \right\|_{\B_{T}} \leq C_T \big(1+\mathrm{b}_{T}^{2\sigma}\big) \left\|q_1-q_2 \right\|_{\B_{T}}
 \]
 and, since $C_T\to0$ as $T\to0$ and $\mathrm{b}_{T}$ is bounded, $\mathcal{G}$ is a contraction on $\B_{T}$, provided that $T$ is small enough.
 
 \emph{Step (iii).} Let $q$ be the solution of \eqref{eq:charge}. From Step (ii), there exists $T_1\in(0,T_*)$ such that $q\in W^{1,1}(0,T_1)$. Now, consider the equation
 \begin{equation}
  \label{eq:charge_new}
  q_1(t) + \int_0^t \dtau \: \bigg(g(t,\tau,q_1(\tau))+\kappa\I(t-\tau) \: q_1(\tau)\bigg) = f_1(t),
 \end{equation}
 where
 \[
  f_1(t):=f(t+T_1)+4\pi\beta\int_0^{T_1}\dtau\:\I(t+T_1-\tau)|q(\tau)|^{2\sigma}q(\tau)-\kappa\int_0^{T_1}\dtau\:\I(t+T_1-\tau)q(\tau).
 \]
 Exploiting Lemma \ref{lem:contr_further} with $T=T_1$ and $h=-4\pi\beta|q|^{2\sigma}q+\kappa q$, one can see that $f_1\in W^{1,1}(0,T)$ for every $T<T_*-T_1$ and arguing as before, there exists $T_1'<T_*-T_1$ and $q_1\in W^{1,1}(0,T_1')$ which solves \eqref{eq:charge_new}. In addition, an easy computation shows that $q(t)=q_1(t-T_1)$ for every $t\in[T_1,T_1+T_1']$, so that we have found a solution to the charge equation such that $q\in W^{1,1}(0,T_1)$ and $q\in W^{1,1}(T_1,T_1+T_1')$, whence $q\in W^{1,1}(0,T_1+T_1')$. This shows that once the regularity is proven up to a time $ T_1 \in (0, T_{*}) $, then it can be extended up to $ T_1 < T_1' < T_{*} $. A priori this procedure could stop before $ T_* $ is reached.
 
 Define $\widehat{T}:=\sup\{T>0:q\in W^{1,1}(0,T)\}$, which is strictly positive by Step (ii). In order to conclude, we must prove that $\widehat{T}=T_*$. Assume, then, by contradiction that $\widehat{T}<T_*$. Consequently, $q\in W^{1,1}(0,T)$ for every $T<\widehat{T}$ and $\|q\|_{L^\infty(0,\widehat{T})} < +\infty$. In addition, fix $\ep>0$ such that $\NN(\widehat{T}-T_\ep)\big(\|q\|_{L^\infty(0,\widehat{T})}^{2\sigma}+1\big)<1/2C$, where $T_\ep:=\widehat{T}-\ep$ and $C$ is a fixed constant that will be specified in the following, and $0<\delta<\ep$, so that $T_\delta:=\widehat{T}-\delta\in(T_\ep,\widehat{T})$. At this point we can estimate $\|q\|_{W^{1,1}(T_\ep,T_\delta)}$ by using \eqref{eq:charge_compact}. First we note that (letting $h$ be defined as before) for $t\in(T_\ep,T_\delta)$
 \[
  q(t)=f(t)-\int_0^{T_\ep}\dtau\:\I(t-\tau)h(\tau)-\int_{T_\ep}^t\dtau\:\I(t-\tau)h(\tau).
 \]
 Since $f\in W^{1,1}(0,T)$ for every $T>0$, its $W^{1,1}(T_\ep,T_\delta)$-norm can be easily estimated independently of $\delta$. The same can be proved for the second term, arguing as in the proof of Lemma \ref{lem:contr_further} and noting that $\I(t-\tau)=\I(t'+T_\ep-\tau)$ with $t'\in[0,T_\delta-T_\ep]$. Summing up, 
 \begin{equation}
  \label{eq:auxxxx}
  \left\|q\right\|_{W^{1,1}(T_\ep,T_\delta)}\leq C_{\widehat{T},T_\ep}+\bigg\|\int_{T_\ep}^{(\cdot)}\dtau\:\I(\cdot-\tau)h(\tau)\bigg\|_{W^{1,1}(T_\ep,T_\delta)}
 \end{equation}
 (precisely, $C_{\widehat{T},T_\ep}$ depends only on $\|q\|_{L^\infty(0,\widehat{T})}$ and $\|q\|_{W^{1,1}(0,T_\ep)}$, which are finite quantities). Therefore, we have to estimate the last term on the r.h.s.. Since the $L^1$ norm can be easily estimated independently of $\delta$, it suffices to study the contribution of the derivative term.
 To this aim we
 note that, for every $t\in(T_\epsilon,T_\delta)$,
 \[
  \frac{\diff}{\diff t}\int_{T_\ep}^t\dtau\:\I(t-\tau)h(\tau)=\I(t-T_\ep)h(T_\ep)+\int_0^{t-T_\ep}\ds \: \I(s)\dot{h}(t-s).
 \]
 First, one has
 \[
  \left|\int_{T_\ep}^{T_\delta}\dt\:\I(t-T_\ep)h(T_{\ep})\right| \leq\NN(\widehat{T}-T_\ep)\|h\|_{L^\infty(0,\widehat{T})}.
 \]
 On the other hand, using Fubini theorem and some changes of variable,
 \[
  \left|\int_{T_\ep}^{T_\delta}\dt\int_0^{t-T_\ep}\ds\I(s)\dot{h}(t-s)\right|\leq \NN(\widehat{T}-T_\ep)\|\dot{h}\|_{L^1(T_\ep,T_\delta)}.
 \]
 Now, easy computations show that
 \[
  \|h\|_{L^\infty(0,\widehat{T})}\leq C\big(\|q\|_{L^\infty(0,\widehat{T})}^{2\sigma}+1\big)\|q\|_{L^\infty(0,\widehat{T})}
 \]
 (see, e.g., \cite[Proof of Proposition 2.4]{CCT}), while, using \cite[Eq. (2.9)]{CCT} and arguing as in the proof of Lemma \ref{lem-lip} (namely, combining \eqref{eq-diff} and \eqref{eq-linf}),
 \[
  \big\|\dot{h}\big\|_{L^{1}(T_\ep,T_\delta)}\leq C \lf(\|q\|_{L^\infty(0,\widehat{T})}^{2\sigma}+1 \ri) \lf\|q\ri\|_{W^{1,1}(T_\ep,T_\delta)}.
 \] 
 Then, recalling \eqref{eq:auxxxx} and the definition of $\ep$ (and possibly redefining $C_{\widehat{T},T_\ep}$), we conclude that
 \[
  \left\|q\right\|_{W^{1,1}(T_\ep,T_\delta)}\leq C_{\widehat{T},T_\ep}+C\NN(T-T_\ep)\big(\|q\|_{L^\infty(0,\widehat{T})}^{2\sigma}+1\big)\left\|q\right\|_{W^{1,1}(T_\ep,T_\delta)}\leq C_{\widehat{T},T_\ep}+\tfrac{1}{2} \left\|q\right\|_{W^{1,1}(T_\ep,T_\delta)}.
 \]
 Hence, moving the last term to the l.h.s., we see that $\left\|q\right\|_{W^{1,1}(T_\ep,T_\delta)}$ can be estimated independently of $\delta$ and thus, letting $\delta \to 0$, there results  $\left\|q\right\|_{W^{1,1}(T_\ep,\widehat{T})}<\infty$. Summing up, we have that $q\in W^{1,1}(0,\widehat{T})$, but, using the first part of Step (iii) with $T_1=\widehat{T}$, this entails that there exists the possibility of a contraction argument beyond $\widehat{T}$, which contradicts the definition of $\widehat{T}$. Hence, we proved that $\widehat{T}=T_*$.
\end{proof}



\subsection{First and Second Derivative of $ \M $}
\label{app-der}

Now, we have all the ingredients to rigorously prove \eqref{derivative M} and \eqref{eq:second_der}.

It is convenient (for the sake of simplicity) to sketch the line of the proof of Proposition \ref{pro:first_der} in advance. First, we introduce the truncated moment of inertia, i.e.
 \begin{equation}
  \label{eq:mom_approx}
  \M_R(t):=\int_{\palla}\dkv\:|\nabk\psitr_t(\kv)|^2, \qquad\forall t\in[0,T_*).
 \end{equation}
Then, we prove that $\M_R$ is differentiable in $[0,T^*)$ (and absolutely continuous in $[0,T]$, $T<T_*$), so that
$$ \M_R (t) = \M_R (0) + \int_0^t \ds \: \dot \M_R (s), \qquad \forall t \in (0,T_*),$$
and, by monotone convergence theorem,
$$ \M (t) = \M (0) + \lim_{R \to \infty}
\int_0^t \ds \: \dot \M_R (s), \qquad \forall t \in (0,T_*).$$
We conclude the proof by applying the dominated convergence theorem, that is proving that
$$
\M(t) = \M(0) + \int_0^t \lim_{R \to \infty} \ds \: \dot \M_R (s),
$$
which implies, therefore, that
\[
\dot\M(t) \ = \ \lim_{R \to \infty} \dot \M_R (t). 
\]
 
\begin{proof}[Proof of Proposition \ref{pro:first_der}]
Let us divide the proof in three steps.

 
 \emph{Step (i).} We start by proving the analogue of \eqref{derivative M} for the truncated moment of inertia $ \M_R $ defined in \eqref{eq:mom_approx}, i.e.
\begin{equation}
  \label{eq:Ie_der}
  \dot{\M}_R(t)=4\:\PI\left\{\int_{\palla}\dkv\:\psitr_t(\kv)\:\kv\cdot\nabk\psitrs(\kv)\right\},\qquad\forall t\in[0,T_*).
 \end{equation}
 
 Preliminarily, integrating by parts in \eqref{eq:psi_stand}, one obtains
 \begin{equation}
  \label{eq:psi_smooth}
  \psitr_t(\kv)=e^{-ik^2t}\phitr_{\lambda,0}(\kv)+\frac{q(0)\,e^{-ik^2t}}{2\pi(k^2+\lambda)}+\frac{i\Q(t)}{2\pi}+\frac{k^2}{2\pi}\int_0^t\dtau\:e^{-ik^2(t-\tau)}\Q(\tau),
 \end{equation}
 with
 \begin{equation}
  \label{eq:q_int}
  \Q(\tau):=\int_0^\tau\ds\:q(s).
 \end{equation}
Hence, 
 \bml{
  \label{eq:psi_grad_smooth}
  \nabk\psitr_t(\kv)=  \, -2it\kv e^{-ik^2t}\phitr_{\lambda,0}(\kv)+e^{-ik^2t}\nabk\phitr_{\lambda,0}(\kv)-\frac{iq(0)t\kv e^{-ik^2t}}{\pi(k^2+\lambda)}-\frac{q(0)\kv e^{-ik^2t}}{\pi(k^2+\lambda)^2} \\
                      +\frac{\kv}{\pi}\int_0^t\dtau\:e^{-ik^2(t-\tau)}\Q(\tau)-\frac{ik^2\kv}{\pi}\int_0^t\dtau\:e^{-ik^2(t-\tau)}(t-\tau)\Q(\tau)	\\
                   =:  \: \Phi_1(t,\kv)+\Phi_2(t,\kv)+\Phi_3(t,\kv)+\Phi_4(t,\kv)+\Phi_5(t,\kv)+\Phi_6(t,\kv)
}
 with $\Q$ defined by \eqref{eq:q_int}, so that one gets
 \[
  \partial_t\big(\nabk\psitr_t(\kv)\big)=-2i\:\kv\:\psitr_t(\kv)-ik^2\:\nabk\psitr_t(\kv),\qquad\forall t\in[0,T_*),\quad\forall\kv\in\R^2.
 \]
Hence, the identity $\partial_t \left(|\nabk\psitr_t(\kv)|^2\right)=2\:\PR\left\{\nabk\psitr_t(\kv)\cdot\partial_t\nabk\psitrs(\kv)\right\}$
yields
 \[
  \partial_t \left(|\nabk\psitr_t(\kv)|^2\right)=4\:\PI\left\{\psitr_t(\kv)\:\kv\cdot\nabk\psitrs(\kv)\right\},\qquad\forall t\in[0,T_*),\quad\forall\kv\in\R^2.
 \]
 In order to bound the difference quotient in $t$ of the integrand of \eqref{eq:mom_approx},  
we use the trivial estimates
 \begin{equation}
  \label{eq:exp}
  \big|e^{-ik^2(t+h)}-e^{-ik^2t}\big|\leq hk^2,
 \end{equation}
 \begin{equation}
  \label{eq:xexp}
  \big|(t+h)e^{-ik^2(t+h)}-te^{-ik^2t}\big|\leq h(h+t)k^2+h,
 \end{equation}
which entail
\beqn
  	\tx\frac{1}{h} \left|\Phi_1(t+h,\kv)-\Phi_1(t,\kv) \right|&\leq& C (k^3 + 1) \big|\phitr_{\lambda,0}(\kv)\big|,	\nonumber \\ 
   \tx\frac{1}{h} \left|\Phi_2(t+h,\kv)-\Phi_2(t,\kv) \right|&\leq& k^2 \big|\nabk\phitr_{\lambda,0}(\kv)\big|,\nonumber\\ 
   \tx\frac{1}{h} \left|\Phi_3(t+h,\kv)-\Phi_3(t,\kv) \right|&\leq& C (k + 1) \lf\|q\ri\|_{L^\infty(0,t)},\nonumber\\ 
   \tx\frac{1}{h} \left|\Phi_4(t+h,\kv)-\Phi_4(t,\kv) \right|&\leq& C \lf\|q\ri\|_{L^\infty(0,t)},\nonumber\\ 
   \tx\frac{1}{h} \left|\Phi_5(t+h,\kv)-\Phi_5(t,\kv) \right|&\leq& C (k^3 + 1) \lf\|q\ri\|_{L^\infty(0,t)},\nonumber\\ 
   \tx\frac{1}{h} \left|\Phi_6(t+h,\kv)-\Phi_6(t,\kv) \right|&\leq& C (k^5 + 1) \lf\|q\ri\|_{L^\infty(0,t)},\nonumber
\eeqn
where each finite constant $ C $ might depend on $ t, h $ and $ \lambda $.
 Summing up,
 \beq
  \label{eq:quo_psi}
  \frac{1}{h} \big|\nabk\psitr_{t+h}(\kv)-\nabk\psitr_t(\kv)\big|\leq  C \lf[ (k^3 +1)\big|\phitr_{\lambda,0}(\kv)\big| + k^2\big|\nabk\phitr_{\lambda,0}(\kv)\big|+ (k^5 + 1) \lf\|q\ri\|_{L^\infty(0,t)} \ri]
\eeq
 and, since
 \[
  \frac{1}{h} \left|\big|\nabk\psitr_{t+h}(\kv) \big|^2- \big|\nabk\psitr_t(\kv)\ri|^2\big| \leq \frac{1}{h} \big| \nabk\psitr_{t+h}(\kv)-\nabk\psitr_t(\kv) \big| \left(\big|\nabk\psitr_{t+h}(\kv)\big| + \big|\nabk\psitr_t(\kv)\big|\right),
 \]
 combining \eqref{eq:quo_psi} with the fact that $\psi_0\in\DES$, one finds that the difference quotient of $|\nabk\psitr_t(\kv)|^2$ is estimated by a function which is integrable in $\palla$, and thus, by dominated convergence, one obtains \eqref{eq:Ie_der}.

\medskip

{\em Step (ii).} We now prove that  
\begin{equation} \label{thesis2}
\lim_{R\to \infty} \int_0^t \ds \: \dot \M_R (s)  \ = \ 
\int_0^t \ds \: \lim_{R \to \infty} \dot \M_R (s).
\end{equation}
More precisely, we find a constant $K > 0 $, possibly depending on $T < T_*$ but independent of $R$, such that $| \dot \M_R (s) | \leq K $ for all $s \in (0,T)$ and then \eqref{thesis2} follows by dominated convergence. 

To this aim, we consider the integrand as the product of the scalar functions $\widehat \psi_s$ and $\kv \cdot \nabk \psitr_s^*$. For the first factor,
we exploit the structure of the energy space and the well-posedness result in \cite{CCT}, that ensures that the solution $\psi_s$ belongs to $V$ at any time $s$, so that it splits as
\begin{equation}
\label{decomposition} 
\psi_s \ = \ \phi_{\lambda,s} + q(s) G_\lambda
\end{equation}
with $\phi_{\lambda,s} \in H^1 (\R^2)$.
For the second factor, by \eqref{eq:gradfdue}, we get
 \begin{multline}
  \label{kgradpsi}
  \kv \cdot \nabk\psitr_s(\kv)=-2isk^2 e^{-ik^2s}\phitr_{\lambda,0}(\kv)+e^{-ik^2s}\kv \cdot \nabk\phitr_{\lambda,0}(\kv) \\
  -\frac{q(s) k^2}{\pi(k^2+\lambda)^2}+ \kv \cdot \nabk\ftr_{1,\lambda}(s,\kv)+\kv \cdot \nabk\ftr_{2,\lambda}(s,\kv).
 \end{multline}
 
We first notice that the pairing of $\widehat \phi_{\lambda,s}$ with $\kv \cdot \nabk \widehat \psi_s^*$ 
is bounded, as the second factor belongs to 
$H^{-1} (\R^2)$, due to \eqref{kgradeffeunodue}. It is then possible to estimate
\[
\left| \int_\palla \diff\kv \, \widehat \phi_{\lambda,s} (\kv) \,
\kv \cdot \nabk \widehat \psi_s^* (\kv) \right|
\ \leq \ \lf\| \phi_{\lambda,\cdot} \ri\|_{L^\infty ((0,T), H^1 (\R^2))} \big\| \kv \cdot \nabk \widehat \psi_\cdot \big\|_{L^\infty((0,T),H^{-1} (\R^2))},
\]
 where the first factor is finite due to  conservation of the energy. The pairing of the charge term $q(s) G_\lambda$ in \eqref{decomposition} with the term $q(s) k^2 / \pi (k^2 + \lambda)^2$ can be understood as a hermitian product in $L^2 (\R^2)$, thus
$$
\left| \int_\palla \diff \kv \, |q(s)|^2 \, \f{k^ 2 \widehat G_\lambda (\kv)}{(k^2 + \lambda)^2} \right| 
\ \leq \ C \lf\| q \ri\|_{L^\infty (0,T)}^2.
$$
It remains to discuss 
the pairing of $\f{q(s)}{k^2 + \lambda}$ with $\kv \cdot \nabk \widehat f_{j, \lambda}^* (s)$. Let us consider $j=2$ only, which gives the most singular term. By \eqref{eq:gradfdue}
 \begin{equation}
  \label{kgradfdue}
  \kv \cdot \nabk\ftr_{2,\lambda}(s,\kv)=\frac{k^2}{\pi(k^2+\lambda)^2}\int_0^s\dtau\:e^{-ik^2(s-\tau)}\dot{q}(\tau)+\frac{ik^2}{\pi(k^2+\lambda)}\int_0^s\dtau\:e^{-ik^2(s-\tau)}(s-\tau)\dot q(\tau).
\end{equation}
Owing to the fact that $\dot q$ belongs to $L^1 (0,T)$, one immediately has that the first term in \eqref{kgradfdue} is square integrable, so we are left to discuss the second one only. To this aim,
we must estimate the integral
$$
\int_\palla \diff\kv \, \f {k^2} {
(k^2 + \lambda)^3} \int_0^s\dtau\:e^{-ik^2(s-\tau)}(s-\tau)\dot q(\tau) $$
by a constant independent of $R$. Using Fubini's theorem and then introducing the variable $u = k^2$, the previous integral reads
\begin{multline*}
\int_0^s \dtau \, (s - \tau) \dot q (\tau)
\int_0^{R^2} \diff u \, \f {u} {(u+ \lambda)^2}
e^{- iu (s - \tau)} \
 = \ 
i \int_0^s \dtau \, \dot q (\tau) \int_0^{R^2} \diff u \, \f{u}{(u + \lambda)^2} \f{\diff}{\diff u} e^{-iu (s - \tau)} \\ 
=  i \int_0^s \dtau \, \dot q (\tau) \left\{  \frac{R^2 e^{- i R^2 (s - \tau)}} { (R^2 + \lambda)^2} 
- \int_0^{R^2} \diff u \, \frac{\lambda - u}{(\lambda + u)^3}
e^{-i u (s - \tau)} 
\right\} \ \leq \ C \lf\| \dot q \ri\|_{L^1 (0,T)}.
\end{multline*}
We can then conclude that every term involved in the integral \eqref{thesis2} can be estimated by a constant, so that \eqref{thesis2} is proved.

{\em Step (iii).} By showing that for any finite $ t > 0  $
\[
\lim_{R\to \infty} \int_0^t \ds \: \dot \M_R (s)   = 4 {\rm{Im}}  \int_0^t \ds \: \int_{\R^2} \diff\kv \, \widehat \psi_s (\kv) \kv 	\cdot \nabla_{\kv} \widehat \psi_s^* (\kv),
\]
we complete the proof of the result. From \eqref{eq:Ie_der}, one has to show that 
$$ \lim_{R \to \infty} \int_{\palla} \diff \kv \:  \widehat \psi_s (\kv) \kv 	\cdot \nabla_{\kv} \widehat \psi_s^* (\kv)$$
exists. 
As in Step (ii), we decompose the integrand into the terms induced by formulae \eqref{decomposition} and \eqref{kgradpsi}, for the two factors $\widehat \psi_s (\kv)$ and
$\kv \cdot \nabla_\kv \widehat \psi_s^*(\kv)$, respectively.

Now, we first observe that
$$ \int_{\R^2} \diff \kv
\big| \widehat \phi_{\lambda,s} (\kv) \, \kv \cdot \nabla_\kv \widehat \psi_s^* (\kv) \big|
\ \leq \ \lf\| \phi_{\lambda,s} \ri\|_{H^1 (\R^2)} \big\| \kv \cdot \nabla_\kv \widehat \psi_s \big\|_{H^{-1} (\R^2)}
$$ 
so that, by monotone convergence, one can conclude that
$$
\lim_{R \to \infty} \int_{\palla} \diff\kv \:  \widehat \phi_{\lambda, s} (\kv) \kv 	\cdot \nabla_{\kv} \widehat \psi_s^* (\kv) \ = \ \int_{\R^2} \diff\kv \:  \phi_{\lambda,s} (\kv) 
\kv 	\cdot \nabla_{\kv} \widehat \psi_s^* (\kv).
$$
Analogously, since
$$
\int_\palla \diff\kv \: \frac{k^2 \widehat G_\lambda (\kv)}{\pi (k^2 + \lambda)^2} \ \leq \ 
\lf\| G_\lambda \ri\|_{L^2 (\R^2)} \left\| \frac{k^2}{\pi (k^2 + \lambda)^2} \right\|_{L^2 (\R^2)},
$$
again by monotone convergence, one gets
$$
\lim_{R \to \infty} | q(s) |^2 \int_\palla \diff\kv \: \frac{k^2 \widehat G_\lambda (\kv)}
{\pi (k^2 + \lambda)^2} \ = \ | q(s) |^2 \int_{\R^2} \diff \kv \: \frac{k^2 \widehat G_\lambda (\kv)}
{\pi (k^2 + \lambda)^2}.
$$
We are thus left to discuss the two terms
$$ \int_\palla \diff\kv \:
\frac{q^*(s)}{k^2 + \lambda} \kv \cdot \nabla_\kv \widehat f_{j, \lambda} (\kv), \qquad j = 1,2.
$$
Like in Step  (ii), we limit ourselves to the term with $j=2$, that is the most singular. From Step (ii), we know that
\begin{multline*}
\lim_{R \to \infty}
\int_\palla \diff\kv \, \f {k^2} {
(k^2 + \lambda)^3} \int_0^s\dtau\:e^{-ik^2(s-\tau)}(s-\tau)\dot q(\tau) \\
= i \int_0^s \dtau \, \dot q (\tau) \left\{  \frac{R^2 e^{- i R^2 (s - \tau)}} { (R^2 + \lambda)^2} 
- \int_0^{R^2} \diff u \, \frac{\lambda - u}{(\lambda + u)^3}
e^{-i u (s - \tau)} 
\right\}
\end{multline*}
and it is immediately seen that the quantity in brackets can be estimated by a constant, so that, by dominated convergence, the limit exists and, by definition of improper integral, one finally has
\bmln{
	\lim_{R \to \infty}
\int_\palla \diff\kv \: \f {k^2} {
(k^2 + \lambda)^3} \int_0^s\dtau\:e^{-ik^2(s-\tau)}(s-\tau)\dot q(\tau)
\\ = 
\int_{\R^2} \diff\kv \: \f {k^2} {
(k^2 + \lambda)^3} \int_0^s\dtau\:e^{-ik^2(s-\tau)}(s-\tau)\dot q(\tau)
}
and this concludes the proof.
\end{proof}

\begin{rem}[Derivative of $ \M $ at $ t = 0 $]
	\label{rem:t=0}
	\mbox{}	\\
	Along the lines of the proof below, one can also show that the derivative of $ \M $ at $ t = 0 $ is in fact well defined and
	\beq\label{eq:M at 0}
		\dot\M(0) = 4\:\PI\bigg\{\int_{\R^2}\dkv\:\psitr_0(\kv)\:\kv\cdot\nabk \widehat{\psi}^*_0(\kv)\bigg\}.
	\eeq
	Indeed, recalling that the regular part of the initial datum $ \psi_0 $ is a Schwartz function, we can easily exchange the limit $ R \to \infty $ with the integral in the expression of $ \dot\M_R(0) $: the latter can be computed using the identity
	\[
		\nabk\psitr_0(\kv)= \nabk\phitr_{\lambda,0}(\kv)-\frac{q(0)\kv}{\pi(k^2+\lambda)^2},
	\]
	which leads to (note the vanishing of a term because of the imaginary part)
	\[
		\dot\M_R(0) = 4\:\PI\bigg\{\int_{k \leq R}\dkv\: \lf( \widehat{\phi}_{\lambda,0}(\kv) + \frac{q(0)}{2\pi(k^2+\lambda)} \ri) \kv\cdot\nabk\widehat{\phi}_{\lambda,0}^*(\kv) -  \frac{q(0)^* k^2 \widehat{\phi}_{\lambda,0}(\kv)}{\pi(k^2 + \lambda)^2} \bigg\},
	\]
	and all the terms are uniformly bounded in $ R $ thanks to the smoothness and decay of $ \phi_{\lambda,0} $, which allows to take the limit $ R \to \infty $ and recover \eqref{eq:M at 0}.	
	\end{rem}



Before showing the proof of \eqref{eq:second_der}, it is necessary to recall a property of compactly supported functions of bounded variation in dimension one.

\begin{lem}
 \label{bv}
 	\mbox{}	\\
 		Let $ q \in C[0,T] $. Then, if $q \one_{[0,T]}\in BV(\R)$ for any $ T < T_* $, one has
 		\beq
 			\label{eq:bv}
 			\bigg|\int_0^t\dtau\:e^{i\rho\tau}q(\tau)\bigg|\leq \frac{C_T}{\rho}, \qquad\forall t \in[0,T],
		\eeq
		for $\rho$ large.
\end{lem}

\begin{proof}
The result is quite classical, but we show the proof for the sake of completeness.
First, note that \eqref{eq:bv} can be rewritten as
 \begin{equation}
  \label{eq:extra-hp_bis}
  \bigg|\widehat{f}_t(-\rho):=\int_\R\dtau\:e^{-i(-\rho)\tau}f_t(\tau)\bigg|\leq \frac{C_T}{\rho}\qquad\forall t\in[0,T],
 \end{equation}
 where $ f_t:=q \one_{[0,t]}\in BV(\R)$ and is compactly supported. Consider, then, a function $ \phi_t\in C_0^\infty(\R)$ such that $0\leq\phi_t\leq1$, $\phi_t\equiv1$ on $[0,t/2]$ and $\mathrm{supp}\{\phi_t\}=[-1,t+1]$. Subsequently, define $ \phi_{t,\rho}  \in C_0^\infty(\R)$ as
 \[
  \phi_{t,\rho}(\tau):=
  \begin{cases}
   \displaystyle \phi_t(\tau),               	& \text{if } \tau \leq t/2,\\ 
   \displaystyle 1,                          		& \text{if }t/2< \tau \leq t_\rho,	\\
   \displaystyle \phi_t(\tau+t/2-t_\rho), & \text{if } \tau>t_\rho,
  \end{cases}
 \]
 where $t_\rho=(t+\rho)/2$. Now,
 \begin{multline*}
  F(\rho,t):=i\rho\int_\R\dtau\:e^{i\rho\tau}f_t(\tau)=i\rho\int_\R\dtau\:e^{i\rho\tau}f_t(\tau)\phi_{t,\rho}(\tau) \\
  +i\rho\int_\R\dtau\:e^{i\rho\tau}f_t(\tau) \lf(1-\phi_{t,\rho}(\tau)\ri)=\underbrace{\int_\R\dtau\:f_t(\tau)\frac{\diff}{\dtau}\big(e^{i \rho \tau}\phi_{t,\rho}(\tau))}_{=:I_1(\rho,t)} \\
  -\underbrace{\int_\R\dtau\:f_t(\tau)\lf( e^{i\rho\tau}\dot{\phi}_{t,\rho}(\tau) \ri)}_{=:I_2(\rho,t)}+\underbrace{i\rho\int_\R\dtau\:e^{i\rho\tau}f_t(\tau)(1-\phi_{t,\rho}(\tau))}_{=:I_3(\rho,t)}.
 \end{multline*}
 One easily sees that
 \[
  |I_3(\rho,t)|\leq \lf\|f_t\ri\|_{L^\infty(0,t)} \rho\int_0^t\dtau\:\lf|1-\phi_{t,\rho}(\tau) \ri| \xrightarrow[\rho \to \infty]{} 0
 \]
 and that
 \[
  |I_2(\rho,t)|\leq \lf\|f_t\ri\|_{L^1(0,t)} \big\|\dot{\phi}_t\big\|_{C^0(\R)}.
 \]
 On the other hand, as $e^{i \rho \tau}\phi_{t,\rho} \in C_0^\infty(\R)$ with sup norm smaller than or equal to one, by the definition bounded variation, one obtains that $ |I_1(\rho,t)|\leq C_t$. Since the procedure above does not depend on the choice of $ t\in[0,T]$ one sees that \eqref{eq:extra-hp_bis} is satisfied.
 \end{proof}
 
 Finally, we can present the proof of Proposition \ref{pro:second_der}.

\begin{proof}[Proof of Proposition \ref{pro:second_der}]
Analogously to the proof of Proposition \ref{pro:first_der}, we first verify the identity on the truncated moment of inertia $ \M_R(t) $ (recall its definition \eqref{eq:mom_approx}) and then we show that the cut-off can be removed. Note that the fact that $ \M \in C^2[0,T] $ follows from the continuity of the r.h.s. of \eqref{eq:second_der}, once the identity is proven.
 
 \emph{Step (i).} First, we compute the partial derivative w.r.t. time of the integrand on the r.h.s. of \eqref{eq:Ie_der}. Setting
 \[
  B(t,\kv):=\PI\left\{4\,\psitr_t(\kv)\:\kv\cdot\nabk\psitrs(\kv)\right\},
 \]
 we have
 \[
  \partial_t B(t,\kv)=4\,\PI\bigg\{\underbrace{\lf[\partial_t\psitr_t(\kv)+ik^2\psitr_t(\kv)\ri]}_{=:A(t)}\kv\cdot\nabk\psitrs(\kv) \bigg\}+8\,k^2\big|\psitr_t(\kv)\big|^2.
 \]
 Therefore, (let $\lambda=1$ throughout) differentiating \eqref{eq:psi_smooth} with respect to time, one sees that $A(t)=\frac{iq(t)}{2\pi}$, so that
 \begin{equation}
  \label{eq:Ie_aux}
  \partial_t B(t,\kv)=\frac{2}{\pi}\PR\left\{q(t)\:\kv\cdot\nabk\psitrs(\kv)\right\}+8\,k^2\big|\psitr_t(\kv) \big|^2.
 \end{equation}
 Since
 \bdm
 	\dot{\M}_R(t)=\int_{\palla}\dkv\:B(t,\kv),
\edm
 it is just left to prove that dominated convergence applies. First, one easily sees that
 \bmln{
   \displaystyle D:= \frac{1}{h}\big|\psitr_{t+h}(\kv)\:\kv\cdot\nabk\psitr_{t+h}^*(\kv)-\psitr_t(\kv)\:\kv\cdot\nabk\psitrs(\kv)\big| \\
     \leq \underbrace{k \big|\nabk\psitr_{t+h}(\kv) \big|}_{=:A_1(t,h)} \: \underbrace{\frac{1}{h} \big|\psitr_{t+h}(\kv)-\psitr_t(\kv)\big|}_{=:A_2(t,h)}+\underbrace{k \big|\psitr_t(\kv) \big|}_{=:A_3(t)} \: \underbrace{\frac{1}{h}\big|\nabk\psitr_{t+h}(\kv)-\nabk\psitr_t(\kv) \big|}_{=:A_4(t,h)}.
}
 Arguing as in the proof of Proposition \ref{pro:first_der} and using \eqref{eq:psi_smooth}, \eqref{eq:psi_grad_smooth}, \eqref{eq:exp} and \eqref{eq:xexp}, one obtains
 \beqn
  	A_1(t,h) &\leq & C k^2 \big| \phitr_{1,0}(\kv) \big|+ k \big|\nabk\phitr_{1,0}(\kv) \big|+ C(k^4+1) \lf\|q \ri\|_{L^\infty(0,t)},
 \nonumber \\[.2cm]
  A_2(t,h)&\leq & k^2 \big|\phitr_{1,0}(\kv) \big| + C(k^4+1) \lf\|q \ri\|_{L^\infty(0,t)},
\nonumber \\[.2cm]
  A_3(t)&\leq & k \big|\phitr_{1,0}(\kv)\big| + C(k^3+1) \lf\|q \ri\|_{L^\infty(0,t)}	\nonumber
\eeqn
 ($A_4$ is already estimated by \eqref{eq:quo_psi}). Hence, since $\psi_0\in\DES$, $D$ is estimated by a function which is Lebesgue integrable and independent of $h$. Thus dominated convergence applies and, combining with \eqref{eq:Ie_aux}, one has
 \begin{equation}
  \label{eq:Ie_aux1}
  \ddot{\M}_R(t)=\frac{2}{\pi}\PR\left\{q(t)\int_{\palla}\dkv\:\kv\cdot\nabk \psitrs(\kv)\right\}+8\int_{\palla}\dkv\:k^2|\psitr_t(\kv)|^2.
 \end{equation}

 \emph{Step (ii).} Now, it is necessary to find a version of \eqref{eq:Ie_aux1}, which makes easier the passage to the limit as $R\to\infty$. First we see that, from the divergence theorem,
 \[
  \int_{\palla}\dkv\:\kv\cdot\nabk\psitrs(\kv)=\int_{\bpalla}\dH\:k\psitrs(\kv)-2\int_{\palla}\dkv\:\psitrs(\kv).
 \]
 On the other hand,
\bmln{
  8\int_{\palla}\dkv\:k^2|\psitr_t(\kv)|^2 = \, 8\int_{\palla}\dkv\:k^2 \big|\phitr_{1,t}(\kv) \big|^2+\frac{8}{\pi}\PR\left\{q(t)\int_{\palla}\dkv\:\frac{k^2\phitr_{1,t}^*(\kv)}{k^2+1}\right\} \\
                                                        +\frac{2|q(t)|^2}{\pi^2}\int_{\palla}\dkv\:\frac{k^2}{(k^2+1)^2}
}
 and, combining with \eqref{eq:Ie_aux1}, we find that
 \bmln{
  \ddot{\M}_R(t) = \, \frac{2}{\pi}\PR\left\{q(t)\int_{\bpalla}\dH\:k\psitrs(\kv)\right\}-\frac{4}{\pi}\PR\left\{q(t)\int_{\palla}\dkv\:\psitrs(\kv)\right\} \\
                  + 8\int_{\palla}\dkv\:k^2|\phitr_{1,t}^*(\kv)|^2+\frac{8}{\pi}\PR\left\{q(t)\int_{\palla}\dkv\:\frac{k^2\phitr_{1,t}^*(\kv)}{k^2+1}\right\}+\frac{2|q(t)|^2}{\pi^2}\int_{\palla}\dkv\:\frac{k^2}{(k^2+1)^2}.
}
 Furthermore, since $\psi_t\in V$, easy computations yield
 \bml{
  \label{eq:Ie_aux3}
  \ddot{\M}_R(t) =  \, 8(C_R(t)-M_R^2(t))+\frac{4}{\pi}\PR\left\{q(t)\int_{\palla}\dkv\:\psitrs(\kv)\right\}-\frac{2|q(t)|^2}{\pi^2}\log(R^2+1) \\
                  +\frac{2}{\pi}\PR\left\{q(t)\int_{\bpalla}\dH\:k\psitrs(\kv)\right\},
}
 where
 \[
  C_R(t):=\int_{\palla}\dkv\:(1+k^2)\big|\phitr_{1,t}(\kv)\big|^2,	\qquad M_R^2(t):=\int_{\palla}\dkv\: \big|\psitrs(\kv)\big|^2.
 \]
 However, one can see that
 \[
  \int_{\palla}\dkv\:\psitr_t(\kv)=J_R(t)+\frac{q(t)}{2}\log(R^2+1),
 \]
 where
\bmln{
  J_R(t) :=  \, \int_{\palla}\dkv\:e^{-ik^2t}\psitr_0(\kv)-\frac{q(0)}{2\pi}\int_{\palla}\dkv\:\frac{e^{-ik^2t}}{k^2+1} \\
             -\frac{1}{2\pi}\int_{\palla}\dkv\:\frac{1}{k^2+1}\int_0^t\dtau\:e^{-ik^2(t-\tau)}(\dot{q}(\tau)-iq(\tau))	
         =:  \, J_{1,R}(t)+J_{2,R}(t)+J_{3,R}(t) 
	}
 and, consequently, \eqref{eq:Ie_aux3} reads
 \bml{
  \label{eq:Ie_aux4}
  \ddot{\M}_R(t) = 8(C_R(t)-M_R^2(t))+\frac{4}{\pi}\PR\{q^*(t)J_R(t)\} +\frac{2}{\pi}\PR\left\{q(t)\int_{\bpalla}\dH\:k\psitrs(\kv)\right\} \\
                 =: \Phi_{1,R}(t)+\Phi_{2,R}(t)+\Phi_{3,R}(t).
	}
 
 \emph{Step (iii).} To complete the proof, we have to take the limit  $R\to\infty$. First, combining monotone convergence with the facts that $\ddot{M}_R$ is bounded and hence Lebesgue integrable on $[0,t]$ and that $\dot{\M}_R(0)\to\dot{\M}(0)$, there results
 \[
  \M(t)=\M(0)+t\dot{\M}(0)+\lim_{R\to\infty}\int_0^t\ds \int_0^s\dtau\:\ddot{\M}_R(\tau).
 \]
 Notice that at this level we do not need to know that $ \M $ is $ C^1 $ but only the convergence of $ \dot{\M}_R(0) $ (see Remark \ref{rem:t=0}). Furthermore, if there exists a continuous function $g(\tau)$ such that
 \begin{equation}
  \label{eq:pass_lim}
  \lim_{R\to\infty}\int_0^t\ds\int_0^s\dtau\:\ddot{\M}_R(\tau)=\int_0^t\ds\int_0^s\dtau\:g(\tau),
 \end{equation}
 then $\ddot{\M}(t)=g(t)$. Consequently, the goal is to compute the l.h.s. of \eqref{eq:pass_lim}, using the decomposition provided by \eqref{eq:Ie_aux4}.
 
 We immediately see, from monotone convergence, that
 \[
  \lim_{R\to\infty}\int_0^t\ds\int_0^s\dtau\:\Phi_{1,R}(\tau)=\int_0^t\ds\int_0^s\dtau\:\left(\|\phi_{1,\tau}\|_{H^1(\R^2)}^2-M^2(\tau)\right).
 \]
 On the contrary, the computation of
 \[
  \lim_{R\to\infty}\int_0^t\ds\int_0^s\dtau\:\Phi_{2,R}(\tau)
 \]
 requires some further efforts. First, we recall that from \cite[Eq. (2.33)]{CCT} (in view of \cite[Eqs. 3.722.1 \& 3.722.3]{GR}), one has
\begin{equation}
 \label{eq-fond}
 \int_{\R^2}\dkv\:\f{e^{-ik^2t}}{k^2+\lambda}=-\pi e^{i\lambda t}\big[\mathrm{ci}(\lambda t)-i\mathrm{si}(\lambda t)\big]=-\pi e^{i\lambda t}\big(\gamma+\log\lambda+\log t-\tf{1}{\pi}Q(\lambda;t)\big)
\end{equation}
for every $\lambda,\, t>0$, where $\mathrm{si}(\cdot)$ and $\mathrm{ci}(\cdot)$ are the sine and cosine integral functions (defined by \cite[Eqs. 5.2.1 \& 5.2.2]{AS}) and
\begin{equation}
\label{eq-q}
 Q(\lambda;t):=-\pi\bigg(\sum_{n=1}^\infty\f{(-t^2\lambda^2)^n}{2n(2n)!}-i\mathrm{si}(\lambda t)\bigg)
\end{equation}
(see, e.g., \cite[Eq. 5.2.16]{AS}).
Hence, we see that, for every $\tau\in[0,t]$,
 \[
  J_{1,R}(\tau)\underset{R\to\infty}{\longrightarrow}J_{1,\infty}(\tau):=2\pi(U_0(\tau)\psi_0)(\zev)
 \]
 and
 \[
  J_{2,R}(\tau)\underset{R\to\infty}{\longrightarrow}J_{2,\infty}(\tau):=\frac{q(0)e^{i\tau}}{2}\left(\gamma+\log\tau-\tfrac{1}{\pi}Q(1;\tau)\right),
 \]
 by definition. In addition, arguing as in the proof of Proposition \ref{pro:first_der}, one finds that
 \[
  \lim_{R\to\infty}\int_0^t\ds\int_0^s\dtau\:\PR\{q^*(\tau)J_{\ell,R}(\tau)\}=\int_0^t\ds\int_0^s\dtau\:\PR\{q^*(\tau)J_{\ell,\infty}(\tau)\},\qquad\ell=1,2.
 \]
 Concerning the third term, we have to prove that
 \begin{equation}
  \label{eq:Jlimtre}
  \lim_{R\to\infty}\int_0^t\ds\int_0^s\dtau\:\PR\{q^*(\tau)J_{3,R}(\tau)\}=\int_0^t\ds\int_0^s\dtau\:\PR\{q^*(\tau)J_{3,\infty}(\tau)\},
 \end{equation}
 with
 \[
  J_{3,\infty}(\tau):=\frac{1}{2}\int_0^\tau\deta\:e^{i(\tau-\eta)}(\dot{q}(\eta)-iq(\eta))\left(\gamma+\log(\tau-\eta)-\tfrac{1}{\pi}Q(1;\tau-\eta)\right).
 \]
 Preliminarily, we observe that by easy computations
 \bml{
  \label{eq:Jtre}
  J_{3,\infty}(\tau)=   \, -\frac{q(0)e^{i\tau}}{2}(\mathrm{ci}(\tau)-i\mathrm{si}(\tau))  -\frac{i}{2}\int_0^\tau\deta\:q(\eta)e^{i(\tau-\eta)}\big(\gamma+\log(\tau-\eta)-\tfrac{1}{\pi}Q(1;\tau-\eta)\big)	\\
                      +\frac{1}{2}\frac{\diff}{\dtau}\int_0^\tau\deta\:q(\eta)e^{i(\tau-\eta)}\big(\gamma+\log(\tau-\eta)-\tfrac{1}{\pi}Q(1;\tau-\eta)\big).
}
 On the other hand, using integration by parts, Fubini theorem and the definitions of $ \mathrm{ci} $ and $ \mathrm{si} $  \cite[Eqs. 5.2.1 \& 5.2.2]{AS}, we obtain that
 \bml{
  \label{eq:Jtre_aux}
  J_{3,R}(\tau)= \, -\frac{q(0)e^{i\tau}}{2}\big(\mathrm{ci}(\tau)-i\mathrm{si}(\tau)\big)+\frac{q(0)e^{i\tau}}{2}\big(\mathrm{ci}(\tau(R^2+1))-i\mathrm{si}(\tau(R^2+1))\big) \\
                  -\frac{q(\tau)}{2}\log(R^2+1)-\frac{1}{2}\int_0^\tau\deta\:q(\eta)\frac{e^{-iR^2(\tau-\eta)}-1}{\tau-\eta}.
 }
 Now, exploiting \eqref{eq-fond} and \eqref{eq-q}, we deduce that
 \[
  \frac{\diff}{\diff y}Q(1;y)=-\pi\frac{e^{-iy}-1}{y}.
 \]
 Consequently,
 \[
  -\frac{1}{2}\int_0^\tau\deta\:q(\eta)\frac{e^{-iR^2(\tau-\eta)}-1}{\tau-\eta}=\frac{i\pi q(\tau)}{4}+\frac{1}{2}\frac{\diff}{\dtau}\int_0^\tau\deta\:q(\eta)\tfrac{1}{\pi}Q(1;R^2(\tau-\eta))
 \]
 and combining with \eqref{eq:Jtre_aux} and (again) with \eqref{eq-fond} and \eqref{eq-q}, there results
 \bml{
  \label{eq:Jtre_auxbis}
  J_{3,R}(\tau)=  \, \underbrace{-\frac{q(0)e^{i\tau}}{2}\big(\mathrm{ci}(\tau)-i\mathrm{si}(\tau)\big)}_{=:\Gamma_{1,R}(\tau)}+\underbrace{\frac{q(0)e^{i\tau}}{2}\big(\mathrm{ci}(\tau(R^2+1))-i\mathrm{si}(\tau(R^2+1))\big)}_{=:\Gamma_{2,R}(\tau)}+\underbrace{\frac{i\pi q(\tau)}{4}}_{=:\Gamma_{3,R}(\tau)}	\\
                \underbrace{-\frac{q(\tau)}{2}\log\left(\frac{R^2+1}{R^2}\right)}_{=:\Gamma_{4,R}(\tau)}+\underbrace{\frac{1}{2}\frac{\diff}{\dtau}\int_0^\tau\deta\:q(\eta)\big(-\mathrm{ci}(R^2(\tau-\eta))+i\mathrm{si}(R^2(\tau-\eta))\big)}_{=:\Gamma_{5,R}(\tau)}	\\
                 +\underbrace{\frac{1}{2}\frac{\diff}{\dtau}\int_0^\tau\deta\:q(\eta)\big(\gamma+\log(\tau-\eta)\big)}_{=:\Gamma_{6,R}(\tau)}.
 }
 In view of \eqref{eq:Jtre_auxbis}, we can finally compute the limit of the r.h.s. of \eqref{eq:Jlimtre}. As $\Gamma_{1,R}$ does not actually depend on $R$, it remains in the limit as $R\to\infty$, while
 \[
  \lim_{R\to\infty}\int_0^t\ds\int_0^s\dtau\:\PR\{q^*(\tau)(\Gamma_{2,R}(\tau)+\Gamma_{3,R}(\tau)+\Gamma_{4,R}(\tau))\}=0.
 \]
 In addition, only using integration by parts and again the properties of $Q(1;\cdot)$, we see that
 \bmln{
  \Gamma_{6,R}(\tau)= \frac{1}{2}\frac{\diff}{\dtau}\int_0^\tau\deta\:q(\eta)e^{i(\tau-\eta)}\big(\gamma+\log(\tau-\eta)\big)-\frac{i}{2}\int_0^\tau\deta\:q(\eta)e^{i(\tau-\eta)}\big(\gamma+\log(\tau-\eta)\big)	\\
                      +\frac{1}{2\pi}\int_0^\tau\deta\:q(\eta)e^{i(\tau-\eta)} \frac{\diff}{\diff \eta}Q(1;\tau-\eta)	\\
                    = \frac{1}{2}\frac{\diff}{\dtau}\int_0^\tau\deta\:q(\eta)e^{i(\tau-\eta)}\big(\gamma+\log(\tau-\eta)\big)-\frac{i}{2}\int_0^\tau\deta\:q(\eta)e^{i(\tau-\eta)}\big(\gamma+\log(\tau-\eta)\big) \\
                      +\frac{i\pi q(\tau)}{4}-\frac{e^{i\tau}}{2\pi}\frac{\diff}{\dtau}\int_0^\tau\deta\:q(\eta)e^{-i\eta} Q(1;\tau-\eta)	\\
                    = \frac{1}{2}\frac{\diff}{\dtau}\int_0^\tau\deta\:q(\eta)e^{i(\tau-\eta)}\big(\gamma+\log(\tau-\eta)\big)-\frac{i}{2}\int_0^\tau\deta\:q(\eta)e^{i(\tau-\eta)}\big(\gamma+\log(\tau-\eta)\big) \\
                      + \frac{i\pi q(\tau)}{4}-\frac{1}{2\pi}\frac{\diff}{\dtau}\int_0^\tau\deta\:q(\eta)e^{i(\tau-\eta)} Q(1;\tau-\eta)+\frac{ie^{i\tau}}{2\pi}\int_0^\tau\deta\:q(\eta)e^{-i\eta} Q(1;\tau-\eta).
}
 Thus, a comparison with \eqref{eq:Jtre} yields that, if one can show that
 \begin{equation}
  \label{eq:gamma_cinque}
  \lim_{R\to\infty}\int_0^t\ds\int_0^s\dtau\:\PR\{q^*(\tau)\Gamma_{5,R}(\tau)\}=0,
 \end{equation}
 then there results that \eqref{eq:Jlimtre} is proved. Now, from an easy computation we find that
 \bmln{
  \int_0^t\ds\int_0^s\dtau\:q^*(\tau)\Gamma_{5,R}(\tau)=  \int_0^t\ds\:q(s)\int_0^s\dtau\:q^*(\tau)\big(-\mathrm{ci}(R^2(s-\tau))-i\mathrm{si}(R^2(s-\tau))\big) \\
                                                        + \int_0^t\ds\int_0^s\dtau\:\dot{q}(\tau)\underbrace{\int_0^\tau\deta\:q^*(\eta)\big(-\mathrm{ci}(R^2(\tau-\eta))-i\mathrm{si}(R^2(\tau-\eta))\big)}_{=:f_R(\tau)}.
	}
 Since the former term can be immediately proved to converge to zero as $R\to\infty$, we only focus on the latter one. However, exploiting \eqref{eq-fond}, \eqref{eq-q}, \cite[Eqs. 5.2.1 \& 5.2.2]{AS} and \cite[Eqs. 3.722.1 \& 2.722.3]{GR}, one can check that (for $R$ large)
 \[
  \lf|\mathrm{si}(R^2(\tau-\eta))\ri|\leq C,	\qquad \lf|\mathrm{ci}(R^2(\tau-\eta))\ri|\leq C(1+\log(\tau-\eta)),\qquad\forall\eta\in[0,\tau).
 \]
Hence, from a repeated use of dominated convergence there results that $f_R\to0$ pointwise and it can be estimated by a bounded function independent of $R$. Since $\dot{q}$ is integrable by Proposition \ref{pro:reg_q}, this implies that \eqref{eq:gamma_cinque} holds true.
 
 Summing up, we have proved that, setting $J_\infty(\tau):=J_{1,\infty}(\tau)+J_{2,\infty}(\tau)+J_{3,\infty}(\tau)$,
 \[
  \lim_{R\to\infty}\int_0^t\ds\int_0^s\dtau\:\Phi_{2,R}(\tau)=\frac{4}{\pi}\PR\left\{\int_0^t\ds\int_0^s\dtau\:q^*(\tau)J_\infty(\tau)\right\}.
 \]
 Now, recalling \cite[Eq. (2.59)]{CCT},
 \bmln{
  \lf(U_0(t)\psi_0\ri)(\zev) = \left(-\beta|q(t)|^{2\sigma}+\frac{\gamma-\log 2}{2\pi}\right)q(t)-\frac{iq(t)}{8}+\frac{q(0)}{4\pi}(-\gamma-\log t) \\
                         +\frac{1}{4\pi}\int_0^t\dtau\:(-\gamma-\log(t-\tau))\dot{q}(\tau)
 }
and arguing as in \cite[Proof of Theorem 1.2]{CCT} (precisely, as in  Part 2.), long and boring computations show that in fact
 \[
  J_\infty(t)=-2\pi\beta|q(t)|^{2\sigma}q(t)+(\gamma-\log 2)q(t).
 \]
 Consequently,
 \bmln{
  \lim_{R\to\infty}\int_0^t\ds\int_0^s\dtau\:\Phi_{2,R}(\tau)=  -8\int_0^t\ds\int_0^s\dtau\:\left(\frac{\beta}{\sigma+1}|q(\tau)|^{2\sigma}+\frac{\gamma-\log 2}{2\pi}\right)|q(\tau)|^2 \\
                                                                -\frac{8\beta\sigma}{\sigma+1}\int_0^t\ds\int_0^s\dtau\:|q(\tau)|^{2\sigma+2}.
 }
 
 Finally, one has to show that
 \begin{equation}
  \label{eq-final_lim}
  \lim_{R\to\infty}\int_0^t\ds\int_0^s\dtau\:\Phi_{3,R}(\tau)=\frac{2}{\pi}\int_0^t\ds\int_0^s\dtau\:|q(\tau)|^2.
 \end{equation}
 First, observe that from \eqref{eq:psi_stand}
 \[
  \int_{\bpalla}\dH\:k\psitr_\tau(\kv)=\underbrace{Re^{-iR^2\tau}\int_{\partial\palla}\dH\:\phitr_{1,0}(\kv)}_{=:A_{1,R}(\tau)}+\underbrace{\frac{R^2e^{-iR^2\tau}q(0)}{R^2+1}}_{=:A_{2,R}(\tau)}+\underbrace{iR^2\int_0^\tau\deta\:q(\eta)e^{-iR^2(\tau-\eta)}}_{=:A_{3,R}(\tau)},
 \]
 so that
 \[
  \int_0^t\ds\int_0^s\dtau\:\Phi_{3,R}(\tau)=\frac{2}{\pi}\PR\bigg\{\int_0^t\ds\int_0^s\dtau\:q^*(\tau)\big(A_{1,R}(\tau)+A_{2,R}(\tau)+A_{3,R}(\tau)\big)\bigg\}.
 \]
 Now, a simple integration by parts of the term involving $A_{3,R}(\tau)$ yields
 \begin{multline*}
  \int_0^t\ds\int_0^s\dtau\:q^*(\tau)\big(A_{1,R}(\tau)+A_{2,R}(\tau)+A_{3,R}(\tau)\big)\\[.2cm] 
  =\int_0^t\ds\int_0^s\dtau\:|q(\tau)|^2+\underbrace{ R^2\bigg(\int_{k = R}\dH\:\phitr_{1,0}(\kv)\bigg)\int_0^t\ds\int_0^s\dtau\:e^{-iR^2\tau}q^*(\tau)}_{=:B_{1,R}(t)} \\
  +\underbrace{\frac{q(0)}{R^2+1}\int_0^t\ds\int_0^s\dtau\:e^{-iR^2\tau}q^*(\tau)}_{=:B_{2,R}(t)}-\underbrace{\int_0^t\ds\int_0^s\dtau\:q^*(\tau)\int_0^\tau\deta\:e^{-iR^2(\tau-\eta)}\dot{q}(\eta)}_{=:B_{3,R}(t)}.
 \end{multline*}
Hence, if one can prove that $B_{j,R}(t)\to0$, as $R\to\infty$, then \eqref{eq-final_lim} is proved. However, it is easy to see that $B_{1,R}(t),\,B_{2,R}(t)\to0$ (for $B_{2,R}(t)\to0$ one uses that fact that $\phi_{1,0}$ is a Schwartz function), whereas
 \bdm
  B_{3,R}(t)=\int_0^t\dtau\:e^{-iR^2\tau}q^*(\tau)(t-\tau)\int_0^\tau\deta\:e^{iR^2\eta}\dot{q}(\eta),
 \edm
 require some further effort. Nevertheless, if
 \[
  \bigg|\int_0^\tau\deta\:e^{iR^2\eta}\dot{q}(\eta)\bigg|\leq C_t, \qquad\forall\tau\in[0,t],
 \]
 or, equivalently, if
 \begin{equation}
  \label{eq:extra-hp}
  \bigg|\int_0^\tau\deta\:e^{iR^2\eta}q(\eta)\bigg|\leq \frac{C_t}{R^2}, \qquad\forall\tau\in[0,t],
 \end{equation}
 as $R\to\infty$, then $B_{3,R}(t)$ vanishes by Riemann-Lebesgue lemma. Now, as shown in Lemma \ref{bv} (set $\rho=R^2$), a sufficient condition for \eqref{eq:extra-hp} is that $q \one_{[0,T]}\in BV(\R)$ for every $T<T_*$, but  this is immediate since $q\in W^{1,1}(0,T)$ for every $T\in(0,T_*)$, by Proposition \ref{pro:reg_q}.
 
 Therefore, \eqref{eq-final_lim} is true and, summing up,
 \[
  \ddot{\M}(t) = \,8\|\phi_{1,t}\|_{H^1(\R^2)}^2-8M^2(t)+8\left(-\beta|q(t)|^{2\sigma}+\frac{2\gamma-2\log 2+1}{4\pi}\right)|q(t)|^2,
 \]
 so that, exploiting the definition of the energy for $\lambda=1$, suitably rearranging terms and using \eqref{eq:energy}, one finds \eqref{eq:second_der}.
\end{proof}

\begin{rem}
 We highlight that the main technical point in proof of the formula of the second derivative of the moment of inertia is the fact that one cannot use the boundary condition \eqref{eq: bc} in the computations, since it is an open issue whether $\psi_t$ is a strong solution of the problem or not.
\end{rem}

\end{document}